\documentclass[11pt, reqno]{article}
\usepackage{amsmath}
\usepackage{epsfig} 
\usepackage{epstopdf} 
\usepackage[colorlinks=true,hypertexnames=false]{hyperref} 
\hypersetup{urlcolor=blue, citecolor=red}
\allowdisplaybreaks

\usepackage{tikz}
\usepackage{mathrsfs}
\usepackage{array}
\usepackage{amsfonts}
\usepackage{amsthm}
\usepackage{amssymb}
\usepackage{graphicx}
\usepackage{esint}
\usepackage{mathtools}
\usepackage{enumitem}
\usepackage[justification=centering]{caption}
\usepackage{float}

\newcommand{\RR}{\mathbb{R}}
\newcommand{\Es}{{\mathcal E}}
\newcommand{\dx}{\,{\rm d}x}
\newcommand{\dy}{\,{\rm d}y}
\newcommand{\dt}{\,{\rm d}t}
\newcommand{\ds}{\,{\rm d}s}

\newcommand{\dtau}{\,{\rm d}\tau}

\newcommand{\A}{\mathcal{L}}

\newcommand{\AM}{\mathcal{L}^{\frac{1}{2}}}
\newcommand{\AI}{\mathcal{L}^{-1}}

\newcommand{\ka}{\overline{\kappa}}

\newcommand{\p}{{\delta_\gamma}}
\newcommand{\G}{\mathbb{G}_\Omega}

\renewcommand{\c}{\mathsf{c}}

\newcommand{\Dc}{D^\alpha_t}

\definecolor{darkblue}{rgb}{0.05, .05, .9}
\definecolor{darkgreen}{rgb}{0.1, .65, .1}
\definecolor{darkred}{rgb}{0.8,0,0}

\newcounter{dummy}
\makeatletter
\newcommand\myitem[1][]{\item[(#1)]\refstepcounter{dummy}\def\@currentlabel{#1}} 
\makeatother

\makeatletter
\newcommand{\labeltext}[3][]{ 
	\@bsphack
	\csname phantomsection\endcsname
	\def\tst{#1}
	\def\labelmarkup{\emph}
	\def\refmarkup{\normalfont}
	\ifx\tst\empty\def\@currentlabel{\refmarkup{#2}}{\label{#3}}
	\else\def\@currentlabel{\refmarkup{#1}}{\label{#3}}\fi
	\@esphack
	\labelmarkup{#2}
}
\makeatother

\newtheorem{theorem}{Theorem}[section]
\newtheorem{corollary}[theorem]{Corollary}

\newtheorem{lemma}[theorem]{Lemma}
\newtheorem{proposition}[theorem]{Proposition}

\theoremstyle{definition}
\newtheorem{definition}[theorem]{Definition}
\newtheorem{remark}{Remark}

\begin{document}
	
    \title{\vspace{-2cm}\bf Time-Fractional Porous Medium Type Equations:\\ Sharp Time Decay and Regularization.}

	\author{\Large Matteo Bonforte$^{\,a}$, Maria Pia Gualdani$^{\,b}$ and
		Peio Ibarrondo$^{\,a}$ \\}
	\date{}

	\maketitle

	\begin{abstract}
We consider  a class of porous medium type of equations with Caputo time derivative. The prototype problem reads as $\Dc u=-\A u^m$ and is posed on a bounded Euclidean domain $\Omega\subset\mathbb{R}^N$ with  zero Dirichlet boundary conditions. The operator $\A$ falls within a wide class of either local or nonlocal operators, and the nonlinearity is allowed to be of degenerate or singular type, namely, $0<m<1$ and $m>1$. This equation is the most general form of a variety of models used to describe anomalous diffusion processes with memory effects, and finds application in various fields, including visco-elastic materials, signal processing, biological systems and geophysical science.

We show existence of unique solution and new $L^p-L^\infty$ smoothing effects.  The comparison principle, which we provide in the most general setting, serves as a crucial tool in the proof and provides a novel monotonicity formula. Consequently, we establish that the regularizing effects from the diffusion are stronger than the memory effects introduced by the fractional time derivative. Moreover, the solution attains the boundary conditions pointwise.

Finally, we prove that the solution  does not vanish in finite time if $0<m<1$, unlike the case with the classical time derivative. Indeed, we provide a sharp rate of decay for any $L^p$-norm of the solution for any $m>0$. Our findings indicate that memory effects weaken the spatial diffusion and mitigate  the difference between slow and fast diffusion.

	\end{abstract}

	\vspace{5mm}
	
	\vspace{4mm}\normalcolor
	

	\vfill
	\begin{itemize}[leftmargin=*]\itemsep2pt \parskip3pt \parsep0pt
        \item[(a)] Universidad Aut\'{o}noma de Madrid, Departamento de Matem\'{a}ticas, \\
        ICMAT - Instituto de Ciencias Matemáticas, CSIC-UAM-UC3M-UCM\\
		Campus de Cantoblanco, 28049 Madrid, Spain.
		
        \item[(b)] The University of Texas at Austin, Mathematics Department\\
        PMA 8.100 2515 Speedway Stop C1200, Austin, TX 78712-1202, USA.

	\end{itemize}

\small
	
	\newpage

	{ \hypersetup{linkcolor=black} \itemsep2pt \parskip1pt \parsep1pt \tableofcontents }
	
	\normalsize
	
	\vspace{5mm}
    \vfill
    \section*{Acknowledgments}

    M.B. and P.I. were partially supported by the Project PID2020-113596GB-I00 (Ministry of Science and Innovation, Spain) and  by the E.U. H2020 MSCA program, grant agreement 777822. M.B. acknowledges financial support from the Spanish Ministry of Science and Innovation, through the ``Severo Ochoa Programme for Centres of Excellence in R\&D'' (CEX2019-000904-S). P.I. was partially funded by the FPU-grant FPU19/04791 from the Spanish Ministry of Universities.
	
     \noindent This work started in Fall 2022 when P.I. was visiting M.G. at  The University of Texas at  Austin, where an essential part of the work was done. P.I. would like to thank both M.G. and the Mathematical Department of The University of Texas at Austin for the kind hospitality, and the above MSCA project for the financial support.

    \noindent M.G. is partially supported by NSF grants DMS-2206677 and DMS-1514761. M.G. also thanks the institute  ICMAT (Madrid, Spain) for the kind hospitality.

	\smallskip\noindent {\sl\small\copyright~2024 by the authors. This paper may be reproduced, in its entirety, for non-commercial purposes.}


\newpage
\section{Introduction}

Several phenomena, from physics to biology to finance, exhibit events during which fractional behavior and memory effects become predominant; for example viscoelastic materials (whose response depends on their current and past states), certain geographical processes including movement of groundwater or transportation through porous media, neuronal and gene regulation networks, but also control theory and more recent modeling of the financial market. Fractional calculus provides a reliable tools to describe memory effects. In 1967 M. Caputo, in
the context of modeling heterogeneous elastic fields, introduced a novel differential operator,
\begin{equation}\label{Caputo derivative}
D^\alpha_t f(t)=\frac{1}{\Gamma(1-\alpha)}\frac{\rm d}{\dt}\int_{0}^{t}\frac{f(\tau)-f(0)}{(t-\tau)^{\alpha}}\dtau,\, \quad \alpha \in (0,1) .
\end{equation}
Operator (\ref{Caputo derivative})  took later the name of Caputo time derivative of order $\alpha $.  In subsequent research, Caputo modeled certain type of fluid diffusing in porous media using this novel nonlocal operator, \cite{Caputo2,Caputo3,Caputo4}. In these geothermal studies, Darcy's Law is adapted to describe fluids that may carry solid particles obstructing the pores, thus diminishing their size and creating a pattern of mineralization. This phenomenon has recently been observed in various other types of porous materials, including building materials \cite{El Abd-constr,El Abd-constr2,Ramos-App} and zeolite \cite{deAz-zeolite,deAz-zeolite2}.

Systems where particles exhibit anomalous diffusion (sub- or super-diffusive behavior) often involve memory effects  \cite{MK-App2,MK-App,SKB-App}. Examples include diffusion in porous media, turbulent flows, and biological transport processes. These applications require mathematical models that allow
particles to do macroscopical long jumps (Lévy flights), leading to the use of nonlocal operators in the spatial and time variables \cite{Applebaum,Jara,MK-App2} .  Many different mathematical models describing anomalous diffusion in a porous medium can be found in the literature. The one we consider here reads as

\begin{equation}\tag{CPME}\label{CPME}
  D_t^\alpha u=-\A u^m\qquad m>0\,,
\end{equation}
and includes a  general class of densely defined operators  $\A$ both of local and nonlocal type.  Equation  \eqref{CPME} is a {\it density dependent diffusion} resulting in a characteristic scaling $x/t^{\alpha/(2+N(m-1))}$, whenever the diffusion operator is $\A=-\Delta$. Its derivation can be found in \cite{Lukasz-deri} and some of its applications can be read here \cite{deAz-zeolite,deAz-zeolite2,El Abd-constr,El Abd-constr2,Ramos-App}. We briefly recall its derivation  in   Section \ref{Physics}.

From a mathematical point of view, the characteristic scaling of the wetting front variable in anomalous diffusion differs from the one of the classical Heat Equation ($x/\sqrt{t}$). Originally, Caputo derivative arose in the linear setting to achieve a subdiffusive characteristic scaling of the form $x/t^{\alpha/2}$ with $\alpha\in(0,1)$. The non-locality in time, or the  memory effect, represents a ``waiting time'' phenomenon typically derived  within the stochastic framework of the Continuous Time Random Walk, see for instance \cite{MK-App2,MK3}. 

In view of the above discussion, \eqref{CPME} encompasses a wide variety of anomalous diffusion models that combine local and nonlocal spatial operators, Caputo fractional time derivative, and  $m$-power like  nonlinearities for any $m>0$. In this manuscript, we provide a comprehensive qualitative and quantitative study of this model.

 Beside global well-posedness, the main contributions of this manuscript are threefold: (i) study of comparison principle and time monotonicity formula, (ii) $L^p-L^\infty$ smoothing effects,  and (iii) optimal long time behaviour. The surprising facts established with our analysis are the {\it regularizing effects }
  and the {\em{non-extinction in finite time}} for all solutions of \eqref{CPME} when $m\in(0,1)$.
 Notably, the memory effect slows down the diffusion, minimizing the relevance of the nonlinearity parameter $m$  in the ranges $m\in (0,1)$ and $m>1$. Nevertheless, the diffusive nature of the equation still provides a regularization of the solution, to the best of our knowledge,  a feature previously unknown for nonlinear equations involving Caputo derivatives.  We emphasize that all these results are new even for the classical Laplacian $\A=-\Delta$ and $\alpha\in(0,1)$.\normalcolor

Within the mathematical framework, the prototype subdiffusive PDE is the ``Heat Equation with memory'',
 \begin{equation}\label{Caputo HE}
   \Dc u=\Delta u\,,
 \end{equation}
 and  the literature on this topic is vaste.
  For the well-posedness and regularity  of \eqref{Caputo HE} in $\mathbb{R}^N$, see for instance the works of Allen, Caffarelli and Vasseur in \cite{ACV}; Kemppainen, Siljander, Vergara and Zacher in \cite{KSVZ}; and Zacher in \cite{Z2008}.
  Optimal asymptotic decay for solutions of \eqref{Caputo HE} are obtained by Cortazar, Quiros and Wolansky in \cite{CQW} for the Cauchy problem in $\mathbb{R}^N$. For the Dirichlet problem on bounded domains, Vergara and Zacher \cite{VZ} show long time decay estimates, also allowing for general operators with variable coefficients in space and time. Recently, Chan, Gomez-Castro and Vázquez studied in \cite{CGV} global well-posedness of \eqref{CPME} with $m=1$ for singular solutions on bounded domains.

 From a nonlinear perspective,  Vergara and Zacher in \cite{VZ}  consider several fractional nonlinear models and obtain sharp decay estimates of $L^q$ norms using fractional ODEs techniques.  The Porous Medium with fractional pressure, associated to the so-called Caffarelli-Vázquez model first introduced in \cite{Caf-Vaz}, has been studied in its Caputo derivative version, by Allen, Caffarelli and Vasseur in \cite{ACV-PME} and by Daus, Xu, Zamponi, Zhang and the second author in \cite{DGXZZ}.

Akagi \cite{GA}, and Li and Liu \cite{LiLiu2019} developed a theory of fractional gradient flows  in  Hilbert spaces, which is the analogous of the Brezis-Komura theory   with fractional time derivative \cite{Brezis-Book-semigr,K67}. This theory provides well-posedness for both linear and for nonlinear problems, and it is the starting point of our paper.


The memory effects complicate the analysis quite a lot. Roughly speaking, the memory effects somehow destroy the semigroup structure,   essential in the De Giorgi-Nash-Moser theory. Surprisingly enough, a nontrivial adaptation of the so-called Green function method, introduced in \cite{BV-PPR1}, allows us to establish our main results. This is achieved by employing novel time monotonicity estimates, a feature that comes as a surprise in the context of the Caputo setting, as we shall further comment below.

To provide further insights about our results, some numerology is in order:  our asymptotic estimates correspond to the known estimates in the formal limit $m\rightarrow 1$. On the other hand, this does not happen in the --formal-- limit $\alpha\rightarrow1^-$. Indeed,  neither the exponents of smoothing effects nor the ones in the long time behaviour estimates, correspond to the known ones when $\alpha=1$.

\vspace{4mm}

Next, we provide a more detailed exposition of our principal findings:\vspace{4mm}

\noindent\textbf{Comparison principle and time monotonicity.} We first recall the results about existence and uniqueness of solution for \eqref{CPME} from \cite{GA,LiLiu2019}. One of the approximations steps proposed in \cite{LiLiu2019} is the starting point of our proof of the comparison principle. The comparison principle here corresponds to ``well-ordering'': if initial data for \eqref{CPME} are ordered, then solutions will preserve the same order for any positive time.

In the present case, we have to overcome two difficulties: we do not have the standard semigroup property $S(t+h)f=S(h)S(t)f$ because of the Caputo derivative, and we lack regularity to perform rigorously the computations. For these reasons, our proof involves several approximation steps.  It is remarkable that we inherit comparison properties from the discretized problem, in such a general framework (see more comments on that after Theorem \ref{Comparison Principle}). This comparison principle lays the cornerstone for several other results,  most importantly for a \textit{new time monotonicity estimates }``a la Benilan-Crandall''.  Namely, we show that $t^{\frac{\alpha}{m-1}}u(t)$ is monotone in time for every $t>0$, cf. Theorem \ref{Thm Benilan-Crandall}. The unexpected nature of this result arises from the memory effect, which complicates any form of time monotonicity.\vspace{4mm}

\noindent\textbf{Smoothing effects and boundary estimates. }We show {\it new $L^p-L^\infty$ smoothing effects}, stated in Theorem \ref{Thm smoothing effect}. To the best of our knowledge, these kind of $L^p-L^\infty$ regularizing effects have been proved by Kemppainen, Siljander, Vergara and Zacher in \cite{KSVZ} only for the linear Caputo Heat equation \eqref{Caputo HE} on the whole space, whenever the initial data is integrable enough. Our regularizing estimates hold for $u_0\in L^p(\Omega)$ with $p>\frac{N}{2}$. There is a structural reason why one should not expect similar results for $p<\frac{N}{2}$: the fundamental solution of \eqref{Caputo HE} in $\mathbb{R}^N$ is unbounded and belongs to $L^q(\mathbb{R}^N)$ with $q<\frac{N}{N-2}$. Therefore, one should ask the initial datum to be in $L^p(\mathbb{R}^N)$ with $p>\frac{N}{2}$ being the conjugate exponent of $q$.

One question for the future is the proof of higher regularity, starting for example with the work of Allen, Caffarelli and Vasseur in \cite{ACV} for H\"older continuity of the linear problem. We do not address this issue in the current manuscript.
The results present in the literature concerning nonlinear diffusions with Caputo time derivative are mostly of propagation of $L^p$-norms: for instance, in \cite{WWZ} the authors prove propagation of $L^\infty$-norm for \eqref{CPME} in the case when $\A=(-\Delta)$, meaning that solutions remains bounded whenever $u_0$ is bounded.

Lastly, we study the behaviour of solutions near the boundary  and show sharp estimates of the form $u(t,x)\lesssim {\rm dist}^{\gamma/m}(x,\partial\Omega)t^{-\alpha/m}$ in Theorem \ref{Thm Boundary behaviour} for some $\gamma\in(0,1]$. Indeed, for ``big initial data'' we obtain lower estimates with matching powers, proving the so-called Global Harnack Principle, that is $u(t,x)\asymp {\rm dist}^{\gamma/m}(x,\partial\Omega)t^{-\alpha/m}$, see Theorem \ref{Thm Boundary behaviour2}. These results are optimal both in space and in time, the latter optimality is explained in what follows.\vspace{4mm}

\noindent{\bf The case of unbounded domains.} All the above results can be carried out, essentially with similar proofs, to the case of \eqref{CPME} posed on unbounded domains, including the whole space $\mathbb{R}^N$. See Remark \ref{Unbounded domains} for more details.

\vspace{4mm}

\noindent\textbf{Optimal long time behaviour.} Our last contribution concerns  a sharp analysis of the long time behaviour, cf. Theorem \ref{Thm Sharp Decay Lp}: in contrast to the local-in-time results, in the fast diffusion regime $m\in (0,1)$, solutions to \eqref{CPME} do not extinguish in finite time. This is an important side effect of the ``memory'' intrinsic in the Caputo derivative. We compute the rates of decay of all $L^p$ norms from above and below, with matching powers. This phenomenon of non-extinction was first observed by Vergara and Zacher in \cite{VZ} for separate-variables solutions.
Here, by means of differential inequalities involving Caputo derivative\footnote{We devote Appendix \ref{sec: FODE} to recall the main results used here about ODE with Caputo derivative.}, we show that the particular behaviour of separate-variables solutions is enjoyed by all the other solutions as well, and is optimal.  This aspect reveals a remarkable discontinuity in the parameter $\alpha$: indeed we show here that solutions decay as $u(t)\asymp t^{-\alpha/m}$, while in the local-in-time case $\alpha=1$ they are known to decay as $u(t)\asymp t^{-1/(m-1)}$ when $m>1$ and extinguish in finite time $u(t)\asymp (T-t)_+^{1/(1-m)}$ when $m\in (0,1)$.

\vspace{4mm}

\noindent\textbf{Related results. }We conclude this section by recalling previous work on the porous medium (PME) and fast diffusion equations (FDE) with the classical time derivative. \normalcolor The literature on these models is extremely rich and spans several decades. The few references we mention here are by no means exhaustive. Both the porous medium and fast diffusion equations have been widely used to describe different physical phenomena, including  flow of gas through porous media \cite{VazBook}, plasma in reactors \cite{BH1978}, etc. For this reason and for their own mathematical interest, they have attracted a lot of attention  since the early 50s, see for instance the lecture notes of Aronson \cite{Aronson-Book} and the books of Barenblatt \cite{B1979}, V\'azquez \cite{VazLN,VazBook}  and Daskalopulos and Kenig \cite{DaskaBook}.  See also the papers \cite{Aronson-Peletier,Vaz-Monath} ($m>1$) and \cite{BF-FDE} ($m<1$).

In the nonlocal case, we refer to the survey \cite{VazSurvey-Cetraro} for a complete review up to 2016, and we just mention the works strictly related to ours. When dealing with the Cauchy problem for  $u_t=-(-\Delta)^s u^m$, a complete theory
has been developed by De Pablo, Quirós, Rodríguez and Vázquez in \cite{DPQR1,DPQR2,DPQRV1,DPQRV2}; see also \cite{BV-ADV, VazJEMS}.
As for the Dirichlet problem in the degenerate case (i.e. $\alpha=1$ and $m>1$ in the present paper), Figalli, Ros-Oton, Sire, Vazquez and the first author, performed a sharp analysis of the behaviour of the solutions in a series of papers \cite{BFR-CPAM,BFV2018CalcVar,BFV2018,BSV2013,BV-PPR1, BV2016}: existence and uniqueness (in optimal classes), comparison, smoothing effects, infinite speed of propagation (contrary to the local case $\A=-\Delta$), local and global Harnack inequalities, sharp interior and boundary regularity estimates, sharp asymptotic behaviour. When $m<1$ the theory is still under development: existence, uniqueness, smoothing effects and boundary estimates have been recently proven by Ispizua, the first and last authors in \cite{BII}, while Harnack and higher regularity estimates shall appear in \cite{BII-TO-DO-IN-PROGRESS}. The proofs rely on the so-called Green function method, introduced in \cite{BV-PPR1}, which offers a valid alternative to the more-classical approaches ``a la'' De Giorgi-Nash-Moser; see \cite[Section 3]{BII} for a comparison between the two methods.

\subsection{Main theorems}

\noindent We study the following Cauchy-Dirichlet problem with Caputo time derivative
\begin{equation}\label{CDP}\tag{CPME}
		\left\{\begin{array}{lll}
			D_t^{\alpha} u(t,x)=-\A u^m(t,x) & \qquad\mbox{on }(0,+\infty)\times\Omega\,,\\
			u(t,\cdot)=0 & \qquad\mbox{on the lateral boundary, }\forall t>0\,,\\
			u(0,\cdot)=u_0  & \qquad \mbox{in } \Omega\,,
		\end{array}\right.
\end{equation}
where $m>0$ and $\Omega\subset \RR^N$ is a bounded smooth domain, with boundary at least $C^{2,\alpha}$ to avoid delicate questions about boundary regularity. We also assume $N>2s$. Note that Caputo derivative is defined as a classical time derivative of a convolution with the kernel $k\in L^1_{\rm loc}((0,+\infty))$:
$$\Dc f(t)=\partial_t[k*(f-f_0)](t)=\frac{\rm d}{ \dt}\int_{0}^{t}k(t-\tau)\,(f(\tau)-f_0)\dtau\,,$$
where the kernel is given by
 $$k(t)=\frac{t^{-\alpha}}{\Gamma(1-\alpha)}\,.$$
This kernel $k$ admits an inverse of the form
 $$\ell(t)=\frac{t^{\alpha-1}}{\Gamma(\alpha)}\,,$$
 with respect to the convolution, i.e., $k*\ell=1$.
It is worth mentioning that there are different types of kernels that define distinct fractional derivatives, and it would be interesting to analyze, see \cite{GA,Stinga,WWZ}.

The operator $\A$ is a linear diffusion operator of local or nonlocal type. We refer to Section \ref{sec: assumptions} and  \ref{sec: examples} for a more precise description and examples. We briefly recall here  the two main assumptions: the inverse operator $\AI: L^2(\Omega)\rightarrow L^2(\Omega)$ is represented via a kernel $\G$:
\[
\AI u(x):=\int_\Omega \G(x,y)u(y)\dy\,.
\]
The Green function $\G$ satisfies
\begin{equation}\label{K1}\tag{G1}
  0\le\G(x,y)\le\frac{c_{1,\Omega}}{|x-y|^{N-2s}}\,,
\end{equation}
for some $s\in(0,1]$. Assumption \eqref{K1} guarantees the compactness of the operator $\AI$ in $L^2(\Omega)$, see \cite[Proposition 5.1]{BFV2018CalcVar}. As a consequence, the operator $\AI$ has a discrete spectrum and a $L^2(\Omega)$ orthonormal basis of eigenfunctions that we denote by $\left(\lambda_k^{-1},\Phi_k\right)_{k\ge 1}$. We moreover assume that $\A$ satisfies the following weak form of Kato's inequality:
\begin{equation}\tag{K}\label{K}
  \int_\Omega\A\left[h(u)\right]\Phi_1\dx\le \int_\Omega h'(u) \A u\,\Phi_1\dx\,,
\end{equation}
 for any convex function $h\in C^1(\mathbb{R})$ with $h(0)=0$, any function $u\in L^1_{\rm loc}(\Omega)$, and $\Phi_1\ge 0$ the first eigenfunction of $\A$.

The solutions that we consider here are defined in the Hilbert space $H^*(\Omega)$, which is the topological dual space of $H(\Omega)$, the domain of the quadratic form associated to the operator $\A$:
$$ H(\Omega)=\left\lbrace u\in L^2(\Omega)\,:\,\int_\Omega u\A u\dx<+\infty\right\rbrace\,.$$
The space $H^*(\Omega)$ is equipped with the following inner product and associated norm:
$$\langle u,v\rangle_{H^*(\Omega)}=\int_\Omega u\,\AI u\dx\qquad\mbox{and}\qquad\|u\|^2_{H^*(\Omega)}=\int_\Omega u\,\AI u\dx\,.$$
Let us recall that when $\A=(-\Delta_\Omega)$ then $H(\Omega)=H^1_0(\Omega)$ and $H^*(\Omega)=H^{-1}(\Omega)$. We refer to  \cite{BII,BSV2013} for more details about the space $H^*(\Omega)$ when $\A$ is a general nonlocal operator.
Now, we introduce the class of solutions that we will consider in this manuscript.
\begin{definition}($H^*$-solutions). Let $u\in L^2((0,T):H^*(\Omega))$, then $u$ is a $H^*$-solution of \eqref{CDP} if the following holds:
\begin{enumerate}[label=\roman*)]
  \item We have $$k*(u-u_0)\in W^{1,2}((0,T):H^*(\Omega)),\qquad [k*(u-u_0)](0)=0,\qquad |u|^{m-1}u\in L^2((0,T):H(\Omega)).$$
  \item For a.e. $t\in(0,T)$
  \begin{equation}\label{strong H solution}
    \int_\Omega\Dc u(t,x)\AI\varphi(x)\dx=-\int_\Omega u^m(t,x)\varphi(x)\dx\qquad\quad\forall\varphi\in H^*(\Omega).
  \end{equation}
\end{enumerate}
\end{definition}
To ensure existence and uniqueness in this class of solutions, we use the theory of fractional gradient flows in Hilbert spaces developed in \cite{GA,LiLiu2019,LiSalgado}. The basic idea is that one would like to perform, in the more delicate framework of Caputo time derivatives, the strategy typical of the Brezis-Komura approach \cite{Brezis-Book-semigr,K67}: gradient flow solutions are obtained  by minimizing at each time step a $\lambda$-convex energy functional, which in the case of the Porous Medium Equation is
\begin{equation*}
  \Es(v)=\frac{1}{1+m}\int_{\Omega}|v|^{1+m}\,.
\end{equation*}
For the Caputo derivative, in \cite{GA} it is proved that the solutions of fractional gradient flows associated to $\Es$ are indeed  $H^*$-solutions according to the above definition.

Our first main result is a comparison principle for  $H^*$-solutions, which has its own interest, and it is also a key ingredient for proving the rest of the results presented in this paper. To the best of our knowledge, this comparison result is new and a complete proof can be found in Section \ref{Ssec.Comparison}, based on the discretized problem studied in Section \ref{Discrete problem}. For the sake of completeness, together with the comparison principle, we state existence and uniqueness result proven by Akagi in \cite{GA} and by Li and Liu in \cite{LiLiu2019}.\normalcolor

\begin{theorem}[Existence, Uniqueness and Comparison Principle]\label{Comparison Principle} Assume \eqref{K1}. Then, for each $u_0\in L^{1+m}(\Omega)\cap H^*(\Omega)$ with $m>0$, there exists a unique $H^*$-solution for \eqref{CDP}. Moreover, let
 $u$ and $v$ be two $H^*$-solutions with initial data $u_0,v_0\in L^{1+m}(\Omega)\cap H^*(\Omega)$, respectively. If assumption \eqref{K} holds and $u_0\le v_0$ a.e. in $\Omega$, then for every $t>0$
\begin{equation}
  u(t)\le v(t)\qquad\mbox{ a.e. in }\,\Omega\,.
\end{equation}
In particular, if $u_0\ge0$ a.e. in $\Omega$, then $u(t)\ge0$ a.e. in $\Omega$ for every $t>0$.
\end{theorem}

%
%

The proof of the comparison result is based on a suitable discretization in the spirit of \cite{LiLiu2019}, together with a clever limiting process: we first prove comparison for the Euler implicit scheme associated to the $H^*$ ``gradient flow structure'' of \eqref{CDP}. We study the discretized problem in Section \ref{Discrete problem}, where we show existence, uniqueness of solution, and a T-contractivity property in a weighted $L^1$ space, cf. Theorem \ref{Discrete T-contractivity L1Phi1}. The T-contractivity estimate at the discretized level, allows us to prove a discrete comparison principle, cf. Corollary \ref{Discrete Comparison Principle}. This last result goes all the way through the limiting process and allows to prove the desired comparison principle in the continuous framework, cf. Section \ref{section: Continuous problem}.
This method diverges from conventional PDE techniques. Typically, the limit solution from a discretization process is obtained through  compactness arguments, and loses the well-ordering property. In our case, however, the uniqueness of gradient flow solution allows to keep  the desired comparison.  Despite the fact that classical solutions and viscosity solutions are the natural classes of solutions where the comparison principle makes sense pointwise \cite{Namba2018}, we prove the result above for solutions in the weak space $H^*(\Omega)\cap L^{1+m}(\Omega)$, which introduces an additional level of difficulty. 

A direct -- and important -- consequence of the comparison principle and of the time-scaling are the Benilan-Crandall time monotonicity formulae that we present below. Let us emphasize that achieving any time monotonicity in PDEs involving the Caputo time derivative is considerably more challenging than with the classical time derivative, due to the ``memory effect'': when $\Dc f(t)\le 0$, we can only deduce a stability estimate, $f(t)\le f(0)$, rather than the ``full'' monotonicity, i.e. $f(t)\le f(t_0)$  for every $t_0 < t$.
\begin{theorem}[Time monotonicity formulae]\label{Thm Benilan-Crandall}
  Assume \eqref{K1} and \eqref{K}, then for any $u$ $H^*$-solution of \eqref{CDP} with nonnegative initial datum $u_0\in L^{1+m}(\Omega)\cap H^*(\Omega)$, it holds that:
  \begin{enumerate}[label=\roman*)]
    \item If $0<m<1$, then the function $$t\mapsto t^{\frac{-\alpha}{1-m}}u(t,x)$$ is non increasing for any $t>0$ and a.e $x\in\Omega$.
    \item If $m>1$, then the function $$t\mapsto t^{\frac{\alpha}{m-1}}u(t,x)$$ is non decreasing for  any $t>0$ and a.e $x\in\Omega$.
  \end{enumerate}
\end{theorem}

Section \ref{section: Smoothing} is devoted to the proof of the new $L^p-L^\infty$ smoothing effects and boundary estimates.  As already mentioned, the only $L^p-L^\infty$ regularizing effects involving the Caputo time derivative present in the literature have been established for the  Heat equation. We refer to \cite[Theorem 3.3.]{KSVZ} in the case of unbounded domains: there, the proof relies on the representation formula and the explicit fundamental solution. For the Heat equation on bounded domains, smoothing effects for the time fractional problem can be inherited from the classical problem with $\alpha=1$, combining the estimates in \cite[Theorem 6.5]{CGV Classical} with \cite[Proposition 2.2.2]{GalWarma}.
 In our setting the proof is much more involved, due to the nonlinearity and the lack of a representation formula.

\begin{theorem}[$L^p-L^\infty$ Smoothing effects]\label{Thm smoothing effect}
  Let $m>0$, $m\neq1$, $N>2s$. Assume that \eqref{K1} and \eqref{K} hold. Given a $H^*$-solution $u$ with initial datum $u_0\in L^{p}(\Omega)\cap  H^*(\Omega)$ with $p>1$ satisfiying
  \begin{equation}\label{Condition for p}
    \begin{cases}
      p\ge 1+m, & \mbox{if } 1+m>\frac{N}{2s}, \\
      p>\frac{N}{2s}, & \mbox{if }\frac{N}{2s}\ge 1+m,
    \end{cases}
  \end{equation}
  then, $u$ satisfies
\begin{equation}\label{Smoothing effects}
  \|u(t)\|_\infty\le\ka\,\frac{\|u_0\|_p^{\frac{1}{m}}}{t^\frac{\alpha}{m}}\qquad\quad\forall t>0\,,
\end{equation}
with $\ka>0$ depending only on $N,s,\alpha,p,m,\Omega$.
\end{theorem}

Surprisingly, the decay rate $t^{-\frac{\alpha}{m}}$ in \eqref{Smoothing effects} is significantly different than in the case $\alpha=1$. \normalcolor For the porous medium equation with classical time derivative, the existence of the Friendly Giant - separate-variable solution of the form $S(x)\,t^{-\frac{1}{m-1}}$ - provides that the decay of the $L^\infty$-norm matches with the time scaling:
 \begin{equation}\label{Friendly Gigant}
    \|u(t)\|_\infty \le K t^{-\frac{1}{m-1}}\qquad\forall t>1,\quad m>1\,,
  \end{equation}
   with a universal constant $K>0$, see \cite{BSV2013,BV-PPR1}.
   In addition, the  smoothing effects for small times provided in \cite{BII,BV-PPR1} for the classical porous medium reads as
\begin{equation}\label{Classical Smoothings}
 \|u(t)\|_\infty\le\ka\,\|u_0\|_p^{2sp\vartheta_p} t^{-N\vartheta_p}\qquad \forall 0<t<1\normalcolor ,\quad\alpha=1,\quad m>0\,,
\end{equation}
  where $u_0\in L^p(\Omega)$ with $p>\frac{N(1-m)}{2s}$ and $\vartheta_p=(2sp-N(1-m))^{-1}$. Interestingly, the decay \eqref{Smoothing effects} corresponds to the decay of solutions to the equation with $\alpha=1$ and $\Omega=\mathbb{R}^N$, see \cite[Theorem 3.1]{BE 2023}.


The condition $p\ge 1+m$ in Theorem \ref{Thm smoothing effect} ensures that solutions are in the energy space needed to define fractional gradient flows, see \cite{GA,LiLiu2019}. On the other hand, the condition $p>\frac{N}{2s}$ seems to be a natural threshold for smoothing effects of \eqref{CDP}. The key point is that the ``memory'' effect makes $u$ behave like the solution of a suitable elliptic problem which always remembers $u_0$,
$$u^m(t,x_0)\sim \int_\Omega \G(x,x_0) u_0(x)\dx\,,$$
where $\G$ represents the Green function of the operator $\A$, see for instance \eqref{Fundamental integral formula}. Thus, the boundedness of $u$ depends directly on the integrability of $\G$, which by \eqref{K1} is $L^q(\Omega)$ with $q<\frac{N}{N-2s}$. This result is also motivated by the explicit solution of the Caputo Heat Equation in $\mathbb{R}^N$, where bounded solutions require $u_0\in L^p(\Omega)$ with $p>\frac{N}{2s}$ due to the unboundedness of the fundamental solution, see \cite{CQW,KSVZ}. This  shows a significant discrepancy between the smoothing effects of Caputo PME and classical PME. We recall that for the latter  we require $u_0\in L^p(\Omega)$ with $p>\frac{N(1-m)}{2s}$, which implies the relaxation of the integrability condition for $u_0$ when $m>0$ gets bigger. See the picture below.
\color{black}
\begin{figure}[H]
\centering
	\begin{tikzpicture}[yscale=0.6,xscale=0.6]
	
    \node[scale=1, draw=cyan,rounded corners,text width=2.5cm]  (Classical) at (5.5,7.7) {Classical PME};	

	\fill[cyan!30,opacity=1, line width=0mm] (0,3) -- (0,6.8) -- (5.5,6.8) -- (5.5,1.7) -- cycle;
	\fill[green!30,opacity=1, line width=0mm] (5.5,1.7) -- (5.5,6.8) -- (10.9,6.8) -- (10.9,1.7)-- cycle;
	
	\node[scale=1.2] (00) at (0,-0.5) {$\mathbf{0}$};
	\node (0) at (-0.5,0){};
	\node (0-1) at (0,-0.3){};
	\node (p-axis) at (0,7){};
	\node (p) at (-0.5,6.4) {p};
	
	\node[scale=1.2] (11) at (9.5,-0.5) {$\mathbf{1}$};
	\node (1,1) at (9.5,2) {};
	\node (m-axis) at (11.1,0) {};
	\node (1) at (9.5,-0.08){};
	\node (1-1) at (9.5,0.3){};
	\node (1-2) at (9.5,-0.3){};
	\node [scale=1.1](m) at (10.75,-0.5){${m}$};
	
	\path[line width=0.6mm,color=darkgreen,dashed,-] (0,3) edge node[pos=0.5,opacity=2,below=1mm,scale=0.9]
        {$\mathbf{p=\frac{N(1-m)}{2s}}$}(5.5,1.7);
	\path[line width=0.5mm,->] (0) edge (m-axis);
	\path[line width=0.5mm,->] (0-1) edge (p-axis);
	\path[line width=0.5mm,-] (1-1) edge (1-2);
	
	\node[scale=1] (N/2s**) at (-0.5,3){$\frac{N}{2s}$};
	\node[scale=1.8] (N/2s*) at (-0.25,5.12){};
	\node[scale=1.8] (N/2s) at (0.02,3){-};
	\node[scale=1.8] (0,1) at (0.02,1.7){-};
	\node(1y) at (-0.5,1.7) {$\mathbf{1}$};
	
	
	\node (mc-axis) at (6,2.06){};
	\node (mc-axis*) at (6,6){};
	
	\node (mc) at (5.5,-0.6){$\frac{N-2s}{N}$};
	\node (mc-1) at (5.5,0.3){};
	\node (mc-2) at (5.5,-0.3){};
	
	\path[line width=0.4mm,-] (mc-1) edge (mc-2);

	\node[color=black,scale=1.2] (LpLpc) at (3,4.5) {$L^p\to L^\infty$};
	\node[color=black,scale=1.2] (L1Lpc) at (7.9,4.5) {$L^1\to L^\infty$};
	
	\draw[dashed,color=blue,opacity=0.4,line width=0.4mm,-] (5.5,0) -- (5.5,6.8);
	\draw[dashed,color=darkgreen,opacity=0.4,line width=0.4mm,-] (9.5,0) -- (9.5,6.8);
	
	\path[line width=0.4mm,color=darkgreen,->] (5.5,1.7) edge (10.9,1.7);
	\path[line width=0.5mm,-] (1-1) edge (1-2);
	\fill[white] (5.5,1.7) circle (0.8mm);
	\draw[darkgreen] (5.5,1.7) circle (0.8mm);


    \node[scale=1, draw=orange!90!black,rounded corners,text width=2.5cm]  (Caputo) at (19.3,7.7) { \;Caputo PME};

    \fill[orange!40,opacity=1, line width=0mm] (13.7,3) -- (13.7,6.8) -- (24.7,6.8) -- (24.7,4)-- (19.7,3) --cycle;
	
	\node[scale=1.2] (00) at (13.7,-0.5) {$\mathbf{0}$};
	\node (0) at (13.2,0){};
	\node (0-1) at (13.7,-0.3){};
	\node (p-axis) at (13.7,7){};
	\node (p) at (13.2,6.4) {p};

	\node (1,1) at (23.7,2) {};
	\node (m-axis) at (24.9,0) {};
	\node (1) at (23.9,-0.08){};
	\node (1-1) at (23.7,0.3){};
	\node (1-2) at (23.7,-0.3){};
	\node [scale=1.1](m) at (24.5,-0.5){${m}$};
	
	\path[line width=0.6mm,color=darkgreen,dashed,-] (13.7,3) edge (19.7,3);
	\path[line width=0.5mm,->] (0) edge (m-axis);
	\path[line width=0.5mm,->] (0-1) edge (p-axis);

	\node[scale=1] (N/2s**) at (13.2,3){$\frac{N}{2s}$};
	\node[scale=1.8] (N/2s*) at (13.45,3.12){};
	\node[scale=1.8] (N/2s) at (13.72,3){-};
	\node[scale=1.8] (0,1) at (13.72,1.7){-};
	\node(1y) at (13.2,1.7) {$\mathbf{1}$};
	
	\node[scale=1.2,color=darkgreen] (pc=1*) at (19.85,1.95) {};
	\node[scale=1.2,color=darkgreen] (pc=1**) at (19.2,1.94) {};
	\node[scale=1.2,color=darkgreen] (pc=1***) at (14.62,2) {};
	
	\node (N/2s-1) at (19.7,-0.6){$\frac{N}{2s}-1$};
	\node (mc-1) at (19.7,0.3){};
	\node (mc-2) at (19.7,-0.3){};
    \path[line width=0.4mm,-] (mc-1) edge (mc-2);

    \node[scale=1.2](1x) at (17,-0.6) {$\mathbf{1}$};
	\node (Ns-1) at (17,0.3){};
	\node (Ns-2) at (17,-0.3){};
    \path[line width=0.4mm,-] (Ns-1) edge (Ns-2);

	\node[color=black,scale=1.4] (LpLpc) at (18.5,5) {$L^p\to L^\infty$};
	
    \draw[dashed,color=red,opacity=0.4,line width=0.4mm,-] (17,0) -- (17,6.8);
	\draw[dashed,color=red,opacity=0.4,line width=0.4mm,-] (19.7,0) -- (19.7,6.8);
	\draw[color=darkgreen,fill=white] (19.7,2) circle;
	
	\path[line width=0.4mm,color=darkgreen,->] (19.7,3) edge node[pos=0.5,opacity=2,below=1mm,scale=1] { $\mathbf{p=1+m}$}(24.7,4);
	\fill[white] (19.7,3) circle (0.8mm);
	\draw[darkgreen] (19.7,3) circle (0.8mm);
	
	\end{tikzpicture}
\caption{On the left side, the smoothing effects \eqref{Classical Smoothings} for the PME with classical time derivative.\\ On the right side, the  smoothing effects \eqref{Smoothing effects} for the PME with Caputo time derivative.}
    \label{fig3}
    \end{figure}
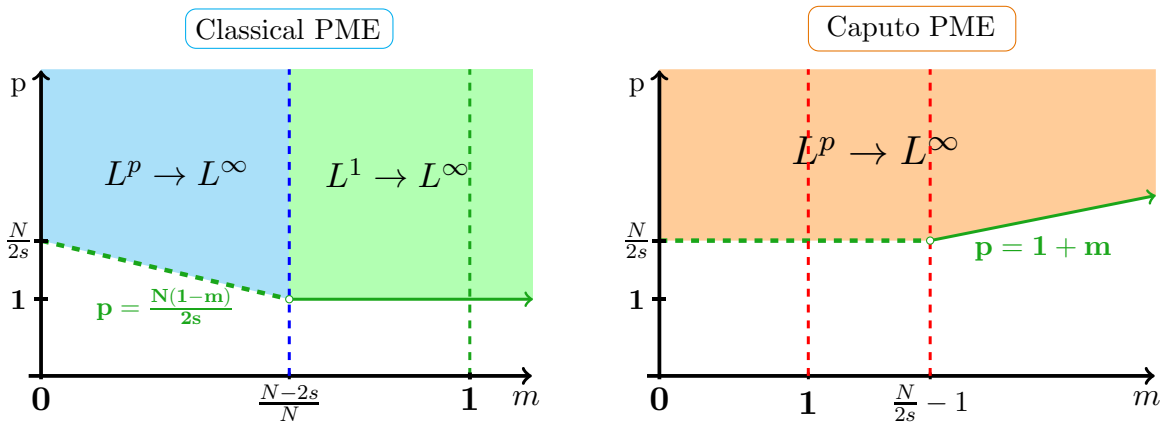

\normalcolor

The proof of Theorem \ref{Thm smoothing effect} is based on the so-called Green function method, introduced in \cite{BV-PPR1} for $m>1$ and in \cite{BII} for $m<1$. The key tool of the proof is  the pointwise bound shown in Proposition \ref{Fundamental pointwise estimate}, which plays the role of an ``almost representation formula''. This pointwise bound is a direct consequence of the dual formulation
\vspace{3mm}
\begin{equation*}
  \AI\Dc u=-u^m\,,
\end{equation*}
and  time monotonicity formula of Theorem \ref{Thm Benilan-Crandall}. This  approach replaces the classical Moser iteration, which seems more challenging due to the memory effect of the time derivative.

Thanks to the Green function method, we are able to show stronger boundary estimates, and in particular we ensure that solutions satisfy the zero Dirichlet boundary condition in a precise and quantitative way. The estimate we provide involves a parameter $\gamma>0$, which characterizes the behaviour of the first eigenfunction of $\A$. The value of $\gamma$ depends on the diffusion operator under consideration, see Section \ref{sec: assumptions} for further explanations.
\begin{theorem}[Upper boundary behaviour]\label{Thm Boundary behaviour}
   Let $m>0$, $m\neq 1$, $N>2s>\gamma$. Assume that \eqref{K2} and \eqref{K} hold. Given a $H^*$-solution $u$ with initial datum $u_0\in L^{p}(\Omega)\cap H^*(\Omega)$ with $p$ satisfying
  \begin{equation*}
    \begin{cases}
      p\ge 1+m, & \mbox{if } 1+m>\frac{N}{2s-\gamma}, \\
      p>\frac{N}{2s-\gamma}, & \mbox{if }\frac{N}{2s-\gamma}\ge 1+m,
    \end{cases}
  \end{equation*}
   then, we have
  \begin{equation}\label{Upper boundary estimates}
  \left\|\frac{u^m(t)}{\mbox{\rm dist}(\cdot,\partial\Omega)^{\gamma}}\right\|_\infty\le\ka\frac{\|u_0\|_p}{t^\alpha},\qquad\quad\forall t>0\,,
\end{equation}
with $\ka$ depending only on $N,s,\alpha,p,m,\Omega$.
\end{theorem}
The  upper bound (\ref{Upper boundary estimates})  is sharp, since at least for ``big initial data'' we provide matching lower bounds:
\begin{theorem}[Global Harnack Principle for ``big'' initial data]\label{Thm Boundary behaviour2}  Under the assumptions of Theorem \ref{Thm Boundary behaviour}, assume moreover that $m>\frac{N-2s}{N+2s}$ and that the initial data also satisfies $u_0(x)\ge \c\,\mbox{\rm dist}(x,\partial\Omega)^\frac{\gamma}{m}$ for a.e. $x\in\Omega$ and some $\c>0$. Then, we have the following Global Harnack Principle
\begin{equation}\label{GHP}
  \frac{c_0}{1+t^{\alpha}}\le\frac{ u^m(t,x)}{\mbox{\rm dist}(x,\partial\Omega)^\gamma}\le \frac{c_1}{t^{\alpha}}\qquad\mbox{ for a.e. } x\in\Omega\quad\forall t>0\,.
\end{equation}
with $c_0,c_1>0$ depending only on $\c,N,s,\alpha,p,m,\Omega$ and $\|u_0\|_p$.
\end{theorem}
The above Global Harnack Principle  \eqref{GHP}, follows by the comparison principle and from the form of the solution(s) of separation of variable.
The proof of the above results is contained in  Section \ref{section: Smoothing}.

The next natural question is whether the asymptotic behaviour of the solution is similar to the one with the classical time derivative. Surprisingly, the answer is negative. We have already commented on the discontinuity in $\alpha$ of the time behaviour, we now state the precise result: all the $L^p$-norms decay in time with the same rate as $t$ goes to infinity, for all $m>0$. In particular, we would like to focus on the fast diffusion regime $m\in (0,1)$, and compare the cases $\alpha=1$ and $0<\alpha<1$. In the latter case,  solutions do not extinguish in finite time,  contrary to the case $\alpha=1$, cf. \cite{BII}, and this is one of the biggest differences between the classical and the Caputo time derivatives.

\begin{theorem}[Sharp Time Decay of $L^p$-norms]\label{Thm Sharp Decay Lp} Assume \eqref{K1} and \eqref{K}. Let $u$ be a $H^*$-solution of \eqref{CDP} with nonnegative initial datum $u_0\in L^{p_0}(\Omega)\cap  H^*(\Omega)$ with $p_0>1$ such that
  \begin{equation*}
    \begin{cases}
      p_0\ge 1+m, & \mbox{if } 1+m>\frac{N}{2s}, \\
      p_0>\frac{N}{2s}, & \mbox{if }\frac{N}{2s}\ge 1+m.
    \end{cases}
  \end{equation*}
Then, for every $1\le p\le\infty$ it holds that
\begin{equation}\label{Sharp Decay Lp}
  \frac{c_0}{1+t^{\frac{\alpha}{m}}}\le \|u(t)\|_p\le\frac{c_1}{t^{\frac{\alpha}{m}}}\qquad\forall t>0\,,
\end{equation}
  where $c_0,c_1>0$ depends on $N,s,\alpha,p,m,\Omega,\,\|u_0\|_{p_0}$.
\end{theorem}
The proof of (\ref{Sharp Decay Lp})  follows by the smoothing effects, abound from above the $L^\infty$-norm, and a lower bound for a weighted $L^1_{\Phi_1}$-norm, inspired by \cite{VZ-Blowup}. More detailed comments and a complete proof can be found in Section \ref{section: sharp decay}.

\begin{remark}[\bf The case of unbounded domains]\label{Unbounded domains}\rm
A closer inspection of the proofs shows that the results proven on  Theorems \ref{Comparison Principle}, \ref{Thm Benilan-Crandall}  and \ref{Thm smoothing effect} can be generalized to unbounded domains including $\mathbb{R}^N_+$ and  $\mathbb{R}^N$ with minimal effort.
 The existence and uniqueness theory of fractional gradient flows, the comparison principle and the time monotonicity carry over essentially unchanged. The smoothing effects follows by decomposing the unbounded domain into a ball of a specific radius and its complement. In this way, it is possible to obtain the same $L^p-L^\infty$ estimate in unbounded domains, following the method in \cite{BE 2023}. Alternatively, following \cite{DPQRV2}, one can use a sequence of solutions to \eqref{CDP} on expanding balls, to which all theorems here apply and provide uniform estimates respect to the size of the ball, and pass to the limit.
 \end{remark}

\subsection{Physical motivation of the model}\label{Physics}

The classical Porous Medium Equation with $\alpha=1$ and $m>1$ is named after its usefulness to describe the behaviour of an ideal gas in a homogeneous porous medium according to the works of Leibenzon \cite{Leibenzon} and Muskat \cite{Muskat}. Similarly, this model also provides a good description of a fluid, typically water, penetrating in a porous medium, such as soil or brick. However, this equation has been used widely to model different types of nonlinear diffusion processes such as nonlinear heat transfer and population dynamics, see \cite{VazBook} and the references therein. We recall from \cite{Lukasz-deri,VazBook} the derivation of the physical model for the groundwater flow in unsaturated porous media.

At a macroscopic level, let us consider $u(t,x)$ to be the fluid concentration at point $x$ at time $t$, and let the vector $\textbf{q}(t,x)$ represent the flux, which denotes the amount and direction of fluid across the unit surface in a unit time. According to the mass conservation property, the function $u$ satisfies the continuity equation
\begin{equation}\label{continuity eq}
  \partial_t u+\nabla\cdot \textbf{q}=0\,.
\end{equation}
In 1856, Darcy \cite{Darcy} showed empirically that the flux $q$ is proportional to the gradient of the pressure
\begin{equation}
  \textbf{q}=-\frac{ \rho k}{\mu}\nabla p\,,
\end{equation}
where $\rho>0$ is the density of the fluid, $\mu>0$ is the viscosity, $k>0$ is the permeability and $p$ is the pressure function. Let us remark that, Darcy's Law can also be derived rigorously by averaging Navier-Stokes equations for fluids in a porous medium \cite{Whitaker}. In most cases and with most fluids, the pressure $p$ can be expressed as a monotone function of the fluid concentration $u$, such that $p=p(u)$ and hence
\begin{equation}\label{Darcy Law}
  \textbf{q}=-\frac{\rho k}{ \mu}\frac{{\rm d} p}{{\rm d} u}\nabla u=:-D(u)\nabla u\,,
\end{equation}
where $D(u)$ is known as the \emph{diffusivity}. Implementing this definition in the continuity equation lead us to the following equation
\begin{equation}\label{Parabolic Darcy Law}
  \partial_t u=\nabla\cdot\left(D(u)\nabla u\right)\,,
\end{equation}
known in the literature as Richards equation \cite{Richards} and Boussinesq's equation \cite{Boussinesq}.
The selection of the pressure $p(u)$ is frequently chosen to fit the experimental results. In particular, one of the most common  is the Brooks-Correy function \cite{Brooks-Corey}
\begin{equation}\label{pressure}
  p(u)\sim u^{m}\,,
\end{equation}
 where the value $m$ varies depending on the soil type under examination and the distribution of the pores within it.

Let us mention that the Fast Diffusion Equation, i.e. $m\in(0,1)$, arises in the gas-kinetics model of Barenblatt-Zel'dovich-Kompaneets \cite{B1979}, and also in plasma physics where high temperatures alter standard physical laws (Fourier's Law).
In this context, the Okuda-Dawson scaling \cite{Okuda} leads to a singular diffusivity $D(u)\sim u^{-1/2}$ in \eqref{Parabolic Darcy Law}. See the pioneering work of Berryman and Holland \cite{BH1978}.

\vspace{2mm}

Recent studies have observed that these models are not accurate enough when addressing certain porous media, such as building materials \cite{El Abd-constr} and zeolite \cite{deAz-zeolite}. In this cases, fluid particles can be trapped in some regions for a significant amount of time delaying the transport of the moisture. Following \cite{Lukasz-deri}, let us show a deterministic derivation of the Caputo Porous Medium Equation for this type of phenomena.

Given a fluid in a porous medium which flow behaves according to \eqref{continuity eq}, let us assume that certain physical phenomena  cause fluid particles to become trapped in specific regions for a duration denoted by $\tau$. As a result, only a part of the instantaneous flux at time $t$ should contribute to the change in the concentration. But we should also consider the released particles that were trapped in the previous time $t-\tau$. Consequently, we can rewrite \eqref{continuity eq} as delayed equation with corresponding weights, denoted by $w_0$ and $w_1$,
\begin{equation*}
  \partial_t u(t)=-\left(w_0\nabla\cdot \textbf{q}(t)+w_1\nabla\cdot \textbf{q}(t-\tau)\right)\,.
\end{equation*}
Following this idea, we can build the model for the case when particles are trapped in different periods of time $0\le\tau_1\le\dots\le \tau_n< t$ with their associated weights $w_1,\dots,w_n$
\begin{equation*}
  \partial_t u(t)=-\sum_{i=1}^{n}w_i\nabla\cdot \textbf{q}(t-\tau_i)\,.
\end{equation*}
By incorporating the weight density $w=w(\tau)$, we can generalize the trapping phenomena to a continuous distribution of periods. Therefore, since $w_i=w(\tau_i)(\tau_i-\tau_{i-1})$ we obtain
\begin{equation*}
  \partial_t u(t)=-\sum_{i=1}^{n}w(\tau_i)\nabla\cdot\textbf{q}(t-\tau_i)(\tau_i-\tau_{i-1})\quad\longrightarrow\quad\partial_t u(t)=-\int_{0}^{t}w(\tau)\nabla\cdot\textbf{q}(t-\tau)\dtau\,,
\end{equation*}
which lead us to the non-local integro-differential equation by a change of variable:
\begin{equation}\label{continuty eq memory}
  \partial_t u(t)=-\int_{0}^{t}w(t-\tau)\nabla\cdot\textbf{q}(\tau)\dtau\,.
\end{equation}
We consider now the family of weight densities $w_\alpha$ with $\alpha\in(0,1)$, satisfying $w_\alpha\rightarrow\delta_0$ in the sense of distributions when $\alpha\rightarrow 1$. Note that when $\alpha\rightarrow 1$ the particles are being trapped for shorter time and the memory-effect is decreasing until equation \eqref{continuity eq} is recovered. As we consider $w_\alpha$ to be a distribution, let us  modified the kernel in \eqref{continuty eq memory} to deal with functions. Using the approximation of Dirac's delta given by  $$w_1=\delta_0=\lim\limits_{h\rightarrow 0}\frac{\chi_{[0,h]}}{h}$$ with the limit in the weak sense, then
\begin{equation*}
  \nabla \textbf{q}(t)=\int_{0}^{t}w_1(t-\tau)\nabla\cdot\textbf{q}(\tau)\dtau=\lim\limits_{h\rightarrow 0}\int_{t-h}^{t}\nabla\cdot\textbf{q}(\tau)\dtau=\frac{\rm d}{\dt}\int_{0}^{t}\nabla\cdot\textbf{q}(\tau)\dtau\,.
\end{equation*}
Hence, we rewrite equation \eqref{continuty eq memory} in terms of a function $k_\alpha$ with $k_1\equiv 1$ as follows
\begin{equation}\label{continuity eq kernel}
  \partial_t u(t)=-\frac{\rm d}{\dt}\int_{0}^{t}k_\alpha(t-\tau)\nabla\cdot\textbf{q}(\tau)\dtau\,.
\end{equation}
It this framework, there are plenty of kernels $k_\alpha$ to be consider, see \cite{GA,Stinga,WWZ}. But, the simplest one is to consider a power function with a normalizing constant $k_\alpha(\tau)=\Gamma(\alpha)^{-1}\tau^{\alpha-1}$,
\begin{equation}\label{RL continuity eq kernel}
  \partial_t u(t)=-\frac{1}{\Gamma(\alpha)}\frac{\rm d}{\dt}\int_{0}^{t}(t-\tau)^{\alpha-1}\nabla\cdot\textbf{q}(\tau)\dtau
  =:-\prescript{}{RL}{D_t^{1-\alpha}} \big(\nabla\cdot\textbf{q}(t)\big)\,.
\end{equation}
Note that we have obtained the Riemann-Liouville fractional derivative in the right hand side. Given suitable regularity conditions, see \cite[Theorem 2.22]{Diethelm},  we can apply Riemann-Liouville integral in both sides of \eqref{RL continuity eq kernel} to obtain the following equation involving the Caputo derivative:
\begin{equation}
 D_t^\alpha u(t)=-\nabla\cdot\textbf{q}(t)\,.
\end{equation}
This equation corresponds to \eqref{CDP} with flux from Darcy's Law \eqref{Darcy Law} and pressure \eqref{pressure}, specifically when $s=1$.

\subsection{Assumptions and functional setting of the operator \texorpdfstring{$\A$}{L}}\label{sec: assumptions}

Let us present the main assumptions on the operator $\A$ and its inverse operator $\AI$. These assumptions  are the standard ones in the literature, see \cite{BFV2018CalcVar,BFV2018,BII,BSV2013,BV-PPR1,BV2016}.

\noindent{\bf Assumptions on $\AI$.} Our approach relies on the dual formulation of \eqref{CDP}, so that we assume the existence of an inverse operator $\AI: L^2(\Omega)\rightarrow L^2(\Omega)$ for $\A$  represented through a kernel $\G$:
\[
\AI u(x):=\int_\Omega \G(x,y)u(y)\dy\,.
\]
The main assumption on the Green function $\G$ is that there exists $c_{1,\Omega}>0$ such that
\begin{equation}\label{K1}\tag{G1}
  0\le\G(x,y)\le\frac{c_{1,\Omega}}{|x-y|^{N-2s}}\,,
\end{equation}
for some $s\in(0,1]$. Assumption \eqref{K1} guarantees the compactness of the operator $\AI$ in $L^2(\Omega)$, see \cite[Proposition 5.1]{BFV2018CalcVar}. As a consequence, the operator $\AI$ has a discrete spectrum and a $L^2(\Omega)$ orthonormal basis of eigenfunctions that we denote by $\left(\lambda_k^{-1},\Phi_k\right)_{k\ge 1}$. Accordingly, we can express
$$\AI u(x)=\sum_{k=1}^{\infty}\lambda_k^{-1}\hat{u}_k\,\Phi_k(x)\,\qquad\mbox{with}\quad \hat{u}_k=\int_\Omega u(y)\Phi_k(y)\dy\,.$$
Furthermore, assumption \eqref{K1} also implies that $\A$ has the discrete spectrum $\left(\lambda_k,\Phi_k\right)_{k\ge 1}$ and then
\begin{equation}\label{spectral L}
  \A u(x)=\sum_{k=1}^{\infty}\lambda_k\hat{u}_k\,\Phi_k(x)\,.
\end{equation}
Hence, assumption \eqref{K1} is enough to construct the Hilbert space $H^*(\Omega)$ -- the functional space where fractional gradient flow theory will be applied -- as the topological dual space of
$$H(\Omega)=\left\lbrace v\in L^2(\Omega)\quad:  \quad \|v\|_{H(\Omega)}^2=\sum_{k=1}^{\infty}\lambda_k\,\left|\hat{v}_k\right|^2<+\infty\right\rbrace\,.$$
 Note that the scalar product and the norm of $H^*(\Omega)$ now read as
$$\langle u,v\rangle_{H^*(\Omega)}=\sum_{k=1}^{\infty}\lambda_k^{-1}\hat{u}_k\,\hat{v}_k \qquad\mbox{and}\qquad  \|u\|^2_{H^*(\Omega)}=\sum_{k=1}^{\infty}\lambda_k^{-1}|\hat{u}_k|^2\,.$$
We refer to  \cite{BSV2013} for a proper identification of the functional spaces $H(\Omega)$ and $H^*(\Omega)$ depending on the specific choice of the operator $\A$. See also \cite{BFV2018CalcVar,BFV2018,BV2016}.

When dealing with boundary issues it is convenient to assume  that the kernel $\G$ satisfies the following condition for some $s,\gamma\in(0,1]$: there exist $c_{0,\Omega}, c_{1,\Omega}>0$ such that for a.e. $x,y\in\Omega$
		\begin{equation}\label{K2}\tag{G2}
			c_{0,\Omega}\p(x)\p(y)\leq\G(x, y)\leq\frac{c_{1, \Omega}}{\lvert x-y\rvert^{N-2s}}
            \left(\frac{\p(x)}{\lvert x-y\rvert^{\gamma}}\wedge 1\right)\left(\frac{\p(y)}{\lvert x-y\rvert^{\gamma}}\wedge 1\right)\,,
		\end{equation}
where $\delta_\gamma(x)=\mbox{dist}(x,\partial\Omega)^\gamma$. It was proven in \cite{BFV2018CalcVar} that \eqref{K2} implies that the first eigenfunction of $\A$, namely, $\A\Phi_1=\lambda_1\Phi_1$ with $\Phi_1\ge 0$ and some constant $\lambda_1>0$, satisfies $\Phi_1(x)\asymp \mbox{dist}(x,\partial\Omega)^\gamma$. Thus, we can rewrite assumption \eqref{K2} as follows
\begin{equation}\label{K3}\tag{G3}
			\c_{0,\Omega}\Phi_1(x)\Phi_1(y)\leq\G(x, y)\leq\frac{\c_{1, \Omega}}{\lvert x-y\rvert^{N-2s}}
            \left(\frac{\Phi_1(x)}{\lvert x-y\rvert^{\gamma}}\wedge 1\right)\left(\frac{\Phi_1(y)}{\lvert x-y\rvert^{\gamma}}\wedge 1\right)\,.
		\end{equation}
Assumptions \eqref{K2} and \eqref{K3} will be used only in Theorem \ref{Thm Boundary behaviour} and Theorem \ref{Thm Boundary behaviour2}. In the rest of the results we just assume \eqref{K1} on $\G$.

\noindent	\textbf{Assumptions on $\A$.} Let $\A:\mbox{\rm dom}(\A)\subset L^2(\Omega)\rightarrow L^2(\Omega)$ be a positive, possibly unbounded and self-adjoint operator. From assumption \eqref{K1}, we know that the operator $\A$ can be expressed in its spectral form as in \eqref{spectral L}, and hence the integration by parts formula follows:
\begin{equation*}
  \int_\Omega u\A v\dx=\int_{\Omega}\AM u\AM v\dx=\int_\Omega v\A u\dx\qquad\forall u,v\in H(\Omega)\,.
\end{equation*}
The main assumption that we need on the operator $\A$ is the following weak form of Kato's inequality: for any convex function $h\in C^1(\mathbb{R})$ with $h(0)=0$ and any function $u\in L^1_{\rm loc}(\Omega)$ with $h(u)\in H^*(\Omega)$, it holds that (recall $\Phi_1\ge 0$ is the first eigenfunction of $\A$)
\begin{equation}\tag{K}\label{K}
  \int_\Omega\A\left[h(u)\right]\Phi_1\dx\le \int_\Omega h'(u) \A u\,\Phi_1\dx\,.
\end{equation}
\begin{remark}
We are not aware of any positive, linear operator satisfying \eqref{K1} which does not satisfy Kato's inequality. Indeed, for the classical Dirichlet Laplacian, $\A=(-\Delta_\Omega)$, Kato's inequality holds pointwise since $h''(r)\ge 0$ when $h$ is a convex function. For nonlocal operators of the form
\begin{equation}\label{nonlocal operators}
  \A u(x)=\int_{\mathbb{R}^N}\left(u(x)-u(y)\right)\mathbb{K}(x,y)\dy+B(x)u(x)\,,
\end{equation}
with $\mathbb{K}(x,y)\ge0$ and $B(x)\ge 0$, Kato's inequality also holds in a stronger form than \eqref{K} as it can be seen in the Appendix of \cite{BII}.

\end{remark}
\vspace{2mm}

\noindent\textbf{Notations. }
We will write the $L^p$ norm in $\Omega$  as $\|f\|_p:=(\int_\Omega|f|^p\dx)^{1/p}$, but for weighted norms with the first eigenvalue $\Phi_1$ of the operator $\A$ we use $\|f\|_{L^p_{\Phi_1}(\Omega)}:=\left(\int_\Omega|f|^p\Phi_1\dx\right)^{1/p}$ in the space $L^p_{\Phi_1}(\Omega):=\lbrace f\in L^1_{\rm loc}(\Omega) : \|f\|_{L^p_{\Phi_1}(\Omega)}<+\infty\rbrace$. The Heaviside function will be denoted by $\mbox{sign}_+(f):=\max\lbrace0, f\rbrace$, and positive and negative part as $(f)_+:=\max\lbrace 0,f\rbrace$ and $(f)_-=\max\lbrace 0,-f\rbrace$, respectively. We will use $f\asymp g$ when there exists $c_1,c_2>0$ such that $c_1 g(x)\le f(x)\le c_2 g(x)$ for a.e. $x$.


\subsection{Examples of \texorpdfstring{$\A$}{operators}}\label{sec: examples}

The assumptions that we state above describe a wide class of operators of local and nonlocal type. Let us enumerate some of the most important ones here.  For a more extended list of examples we refer to \cite{BII} and \cite[Section 3]{BV2016}.\vspace{2mm}

\noindent{\bf The Dirichlet Laplacian.} We can consider the classical Dirichlet Laplacian $\A=(-\Delta_\Omega)$ with $s=1$, where the Green function $\G$ is bounded by the classical Newtonian potential. For this operator \eqref{K2} is satisfied with $\gamma=1$.\vspace{2mm}

The main examples of nonlocal operators on bounded domains are the three fractional Laplacians that we state below. Note that, on the whole space $\mathbb{R}^N$, the three definitions are equivalent, but on bounded domains they represent entirely different nonlocal operators, as highlighted by the behaviour of the first eigenfunction for each of them. Indeed, from a probabilistic standpoint, these operators are the infinitesimal generator of distinct Lévy processes.\vspace{2mm}

\noindent\textbf{The Restricted Fractional Laplacian (RFL). }We define this fractional Laplacian operator acting on a bounded domain by using the integral representation on the whole space in terms of a hypersingular kernel, namely,
\begin{equation}\label{RFL}
  (-\Delta_{|\Omega})^sf(x)=c_{N,s}\;\mbox{P.V. }\int_{\RR^N}\frac{f(x)-f(y)}{|x-y|^{N+2s}}\dy\,,
\end{equation}
where $s\in(0,1)$ and we restrict the operator to functions that are zero outside $\Omega$. For this operator \eqref{K2} is satisfied with $\gamma=s$.\vspace{2mm}

\noindent\textbf{Spectral Fractional Laplacian (SFL). }Consider the so-called \emph{spectral definition} of the fractional power $s\in(0,1)$ of the classical Dirichlet Laplacian $\Delta_\Omega$ on $\Omega$ defined by a formula in terms of semigroups, namely
\begin{equation}\label{SFL}
  (-\Delta_{\Omega})^sf(x)=\frac{1}{\Gamma(-s)}\int_{0}^{\infty}(e^{t\Delta_{\Omega}}f(x)-f(x))\frac{\dt}{t^{1+s}} =\sum_{j=1}^{\infty}\lambda_j^s\hat{f}_j\phi_j,
\end{equation}
where $(\lambda_k,\phi_k)_{k\ge 1}$ are the eigenvalues and eigenfunctions of the classical Dirichlet Laplacian on $\Omega$, $\hat{f}=\int_{\Omega}f(x)\phi_k(x)\dx$, and $\|\phi_k\|_2=1$. Since the eigenvalues are the same as in the Dirichlet Laplacian, \eqref{K2} is fulfilled with $\gamma=1$. \vspace{2mm}

\noindent\textbf{Censored Fractional Laplacian (CFL). }
This third kind of Fractional Laplacian was introduced in \cite{bogdan-censor}, in connection with censored stable processes. The operator takes the form:
\begin{equation}\label{CFL}
  \A f(x)=\mathrm{P.V.}\int_{\Omega}\frac{f(x)-f(y)}{|x-y|^{N+2s}}\dy\,,\qquad\mbox{with }\frac{1}{2}<s<1\,.
\end{equation}
The restriction on the parameter $s$ is due to the lack of trace inequality for $s\in (0,1/2]$. This operator satisfies \eqref{K2} with $\gamma=2s-1$.\vspace{2mm}

\noindent\textbf{Other examples:} As explained in \cite[Section 3]{BV2016}, our theory may also be
applied to: i) the operators above with uniformly elliptic, $C^1$ and symmetric coefficients; ii) sums of two fractional operators; iii) the sum of the classical Laplacian and a nonlocal operator kernel; iv) Schrödinger equations for non-symmetric diffusions; v) gradient perturbation of restricted
fractional Laplacians.

\vspace{2mm}

\noindent\textbf{Functional spaces associated to $\A$} We recall that the three fractional Laplacians on domains, \eqref{RFL}, \eqref{SFL} and \eqref{CFL}, can be expressed as in \eqref{nonlocal operators}, see Section \ref{sec: examples} and \cite{BV2016} for more details and examples. Moreover, for \eqref{RFL}, \eqref{SFL} and \eqref{CFL}, it is possible to identify $H(\Omega)$ and $H^*(\Omega)$ with the standard fractional Sobolev spaces as follows: Let $s\in(0,1]$
  \begin{equation*}
  H(\Omega)=
    \begin{cases}
      H_0^s(\Omega), & \mbox{if } s>1/2 \\
      H^{1/2}_{00}(\Omega), & \mbox{if } s=1/2 \\
      H^s(\Omega), & \mbox{if } s<1/2,
    \end{cases}
    \qquad\mbox{and}\qquad
    H^*(\Omega)=
    \begin{cases}
     H^{-s}(\Omega), & \mbox{if } s>1/2 \\
      \left(H^{1/2}_{00}(\Omega)\right)^*, & \mbox{if } s=1/2 \\
      H^{-s}(\Omega), & \mbox{if } s<1/2.
    \end{cases}
  \end{equation*}
  Recall that $H_0^s(\Omega)=H^s(\Omega)$ when $s\le1/2$, and hence the dual coincide. If $s=1/2$, then $H^s_{00}(\Omega)\subsetneq H^s_0(\Omega)$, which implies $H^{-s}(\Omega)\subsetneq\left(H^s_{00}(\Omega)\right)^*$. See \cite[Sec. 7.7]{BSV2013} for more details. This identification holds for any operator $\A$ with discrete spectrum such that $0<\lambda_{k+1}/\lambda_k\le\Lambda_0$ for some $\Lambda_0>0$.

\vspace{2mm}

\section{Discretized problem}\label{Discrete problem}
In this section we will provide some properties of the discretization of \eqref{CDP} and then deduce the comparison principle. Several discretizations of the Caputo derivative have been proposed through the years, see for instance \cite{DGXZZ,LiLiu2019,LiSalgado}. The discretization that works best for our convergence analysis is the one proposed in \cite{LiLiu2019}, which is based on a deconvolution scheme.

\begin{definition}\cite[Section 3.]{LiLiu2019}
  Let $\tau=T/N>0$ be the time-step and $t_k=k\tau$ with $0\le k\le N$ points of the interval $(0,T)$.
  Let us consider the Euler implicit scheme for \eqref{CDP}:
\begin{equation}\label{DCDP}
  \begin{cases}
    (D^\alpha_\tau U)_{k}=-\A U_{k}^m&\mbox{in}\quad\Omega,\quad \forall k\ge 0\,,\\
    U_0=u_0&\mbox{in}\quad\Omega\,,
  \end{cases}
\end{equation}
where
  \begin{equation}\label{Discrete Caputo derivative}
  \left(D^\alpha_\tau U\right)_{k}:=\frac{1}{\tau^\alpha}\left(c_k^k(U_{k}-U_0)+\sum_{i=1}^{k-1}c_i(U_{k}-U_{k-i})\right)\,,
\end{equation}
  with $c_k^k:=\sum_{i=k}^{\infty}c_i>0$ and $c_i>0$ are defined recursively by
  $$
  c_i:=\sum_{j=0}^{i-1}\left((i+1-j)^\alpha-(i-j)^\alpha\right)c_j\qquad \mbox{for any }i\ge1\,,
  $$
  with $c_0:=\Gamma(1+\alpha)=c_k^k+\sum_{i=1}^{k-1} c_i$. Equivalently, we also will use
   \begin{equation}\label{Discrete caputo derivative 2}
    (D^\alpha_\tau X)_k=\frac{1}{\tau^\alpha}\left(c_0(X_{k+1}-X_0)-\sum_{i=1}^{k-1}c_i(X_{k-i}-X_0)\right)\,.
  \end{equation}
\end{definition}
One of the drawbacks of the Caputo derivative is the lack of a chain rule. However it is noted in \cite{GA,WWZ}, that if the kernel $k$ belongs to $H^1([0,T])$, then the ``chain rule inequality'' below holds for convex functions $h\in C^1(\mathbb{R}^N)$, see Lemma \ref{fundamental identity}. Namely, we have that
$$\partial_t[k*(h(f)-h(f_0))](t)\le h'(f(t))\partial_t[k*(f-f_0)](t)\,.$$
 Another remarkable difference with the classical time derivative is the fact that $\Dc f(t)\le 0$ only implies $f(t)\le f(0)$, but not the monotonicity $f(t)\le f(t_0)$ for all $t\ge t_0$.

Also  the discrete Caputo derivative (\ref{Discrete Caputo derivative}) satisfies the  monotonicity formula and the chain rule inequality:
\begin{proposition}\cite[Theorem 3.1.]{LiLiu2019}\label{convex chain rule}
  \begin{enumerate}[label=\roman*)]
  \item (Chain rule inequality). Let $h:\mathbb{R}\rightarrow\mathbb{R}$ be a convex function in $C^1(\mathbb{R})$. Then, for every sequence of functions $\lbrace U_k\rbrace_{k\ge 0}\subset L^1_{loc}(\Omega)$ it holds that
  \begin{equation*}
    (D^\alpha_\tau \,h(U))_k\le h'(U_k)\,(D^\alpha_\tau U)_k\qquad\mbox{a.e. in }\Omega\,.
  \end{equation*}
  \item (Nonpositive discrete Caputo derivative). Assume that a sequence $\lbrace X_k\rbrace_{k\ge 0}\subset \mathbb{R}$  satisfies $(D^\alpha_\tau X)_k\le 0$
  for every $k\ge 1$, then
  \begin{equation*}
  X_k\le X_0\qquad\forall k\ge 1\,.
  \end{equation*}
  \end{enumerate}
\end{proposition}
\begin{proof}
\begin{enumerate}[label=\roman*)]
  \item  Since $h\in C^1(\mathbb{R})$ is convex , we have that $h'(x)(x-y)\ge(h(x)-h(y))$ for every $x,y\in\mathbb{R}$. Then,
  \begin{align*}
    h'(U_k)(D^\alpha_\tau u)_k &  =\frac{1}{\tau^\alpha}\left(h'(U_k)(U_k-U_0)c_k^k+\sum_{i=1}^{k-1}h'(U_k)(U_{k}-U_{k-i})c_i\right)\\
    &\ge \frac{1}{\tau^\alpha}\left(c_k^k\left(h(U_k)-h(U_0)\right)+\sum_{i=1}^{k-1}c_i(h(U_{k})-h(U_{k-i}))\right).
  \end{align*}
  \item We prove the result by induction. For the case $k=1$ is trivial since $\tau,c_0>0$. Now, assume that $X_k\le X_0$ for every $k=1,\dots n$, and let us prove it for $k=n+1$. From \eqref{Discrete caputo derivative 2} and $(D^\alpha_c X)_{k+1}\le 0$ we know that
  \begin{equation*}
     (D^\alpha_\tau X)_{k+1}=\frac{1}{\tau^\alpha}\left((c_0(X_{k+1}-X_0)-\sum_{i=1}^{k}c_i(X_{k+1-i}-X_0)\right)\le 0\,,
  \end{equation*}
  which by the induction hypothesis give us
  $$c_0(X_{k+1}-X_0)\le\sum_{i=1}^{k}c_i(X_{k+1-i}-X_0)\le 0\,,$$
   since $c_i>0$ for every $i\ge 0$.
\end{enumerate}

\end{proof}

\subsection{Existence and uniqueness for the discretized problem}

The method used to prove existence of solutions for \eqref{DCDP} is a time-fractional version of the classical minimizing movements scheme \cite{LiLiu2019,LiSalgado}. As in the case of classical time derivative \cite{BII}, the energy used to construct the gradient flow solutions is
\[
\Es(u):=\begin{cases}
           \frac{1}{1+m}\int_\Omega |u(x)|^{1+m}\dx, & \mbox{if } u\in L^{1+m}(\Omega),  \\
           +\infty, & \mbox{otherwise}.
         \end{cases}
\]
Notice that this energy functional is convex and lower semicontinuous in the Hilbert space $H^*(\Omega)$, which is the (topological) dual space of $H(\Omega)$, the domain of the quadratic form associated to $\A$,
$$H(\Omega)=\left\lbrace u\in L^2(\Omega)\; :\; \|u\|^2_{H(\Omega)}=\int_\Omega u\A u\dx<+\infty\right\rbrace\,.$$
Let us recall that the space $H^*(\Omega)$ is a Hilbert space endowed with the scalar product and the norm
$$\langle u, v\rangle_{H^*(\Omega)}=\int_\Omega u\AI v\dx \qquad\mbox{and}\qquad \|u\|^2_{H^*(\Omega)}=\int_\Omega u\AI u\dx\,.$$
The classical gradient flow theory in Hilbert spaces -- in particular, in $H^*(\Omega)$ -- provides solutions of
\begin{equation}\label{General GF}
  -\partial_t u(t)\in\partial\Es(u(t))\qquad\mbox{for a.e. }t>0\,,
\end{equation}
where the subdifferential of $\Es$ is defined by
$$\xi\in\partial\Es(u)\quad\Leftrightarrow\quad\Es(v)-\Es(u)\ge\langle\xi,v-u\rangle_{H^*(\Omega)}\qquad\forall v\in H^*(\Omega)\,.$$
In our case, note that \eqref{General GF} is related to the classical PME on domains; note in fact that for any $u$ such that
\begin{equation*}
  u\in D(\partial\Es)=\lbrace v\in L^{1+m}(\Omega)\cap H^*(\Omega)\;:\; |v|^{m-1}v\in H(\Omega)\rbrace\,,
\end{equation*}
we have, in the $H^*(\Omega)$ sense,
$$\xi\in\partial\Es(u)\qquad\Leftrightarrow\qquad\AI\xi=|u|^{m-1}u\,.$$

In the fractional gradient flow theory, the classical time derivative $\partial_t u$ in \eqref{General GF} is substituted by a fractional time derivative, such as the Caputo derivative, $\Dc u$. The well-posedness of fractional gradient flow solutions in Hilbert spaces has been proven in \cite{GA,LiLiu2019,LiSalgado}. The first step in this approach is to provide the existence and uniqueness of the time discretized equation \eqref{DCDP}, that we stated below.

\begin{theorem}\label{Existence of discrete GF}\cite{LiLiu2019,LiSalgado}
  Let $m>0$ and assume \eqref{K1}. For any $U_0\in L^{1+m}(\Omega)\cap H^*(\Omega)$ fixed there exists a unique sequence of functions $\lbrace U_k\rbrace_{k\ge 0}\subset D(\partial\Es)$ satisfying
  \begin{equation}\label{Discrete Gradient Flow}
    -(D^\alpha_\tau U)_{k}\in \partial\Es(U_{k})\qquad\forall k \ge 0\,,
  \end{equation}
  which is \eqref{DCDP} in the $H^*$-sense, namely, for every $k\ge 0$
  \begin{equation}\label{Discrete H* formulation}
    \int_\Omega (D^\alpha_\tau U)_{k}\AI \varphi\dx=-\int_\Omega U_{k}^m\varphi\dx\qquad\forall \varphi\in H^*(\Omega)\,.
  \end{equation}
   Moreover,
  \begin{equation}\label{discrete Decay of the Energy}
     \Es(U_k)\le\Es(U_0)\qquad\forall k\ge 1\,.
   \end{equation}
\end{theorem}
The idea of the proof is the same of the nonlinear theory of Brezis for gradient flows \cite{Brezis-Book-semigr}, see also Komura \cite{K67}, the celebrated JKO scheme \cite{JKO} and the excellent lecture notes \cite{ABS}. Due to the convexity of $\Es$ we can construct recursively the next function $U_{k+1}$ with $k\ge 0$ by the minimizing movements scheme, namely, for any $k\ge 0$
\begin{equation}\label{Minimizing movement scheme}
  U_{k+1}:=\underset{u\in H^*(\Omega)}{\mbox{argmin}} \left[\frac{1}{2\tau^\alpha}\left(c_{k+1}^{k+1}\|u-U_0\|^2_{H^*(\Omega)}+\sum_{i=1}^{k}c_i\|u-U_{n-i}\|_{H^*(\Omega)}^2\right)+\Es(u)\right]\,.
\end{equation}
By the lower semicontinuity and convexity of the energy functional the sequence $ U_{k}$ is unique and satisfies \eqref{Discrete Gradient Flow}. The decay of the energy \eqref{discrete Decay of the Energy} follows from Proposition \ref{convex chain rule}, namely,
\begin{align*}
  (D^\alpha_\tau \Es(U))_{k+1}&\le\int_\Omega U^{m}_{k+1}(D^\alpha_\tau U)_{k+1}\dx =-\int_\Omega U_{k+1}^m\,\A U^m_{k+1}\dx=-\|U_{k+1}^m\|^2_{H^*(\Omega)}\le 0\,.\\
\end{align*}

\subsection{Discrete Comparison Principle}
The main result of this section is the proof of the comparison principle for the discrete solutions obtained in Theorem \ref{Existence of discrete GF}. For this purpose, we first prove a  T-contractivity estimate in the Banach space $L^1_{\Phi_1}(\Omega)$. Recall that assumption \eqref{K2} implies that the operator $\A$ has a discrete spectrum $(\lambda_k,\Phi_k)_{k\ge 1}$ with $\Phi_k\in L^\infty(\Omega)$. Indeed, eigenfunctions are classical in the interior, i.e. $C^{2s+\gamma}$, whenever the operator allows it, and have sharp boundary regularity $C^\gamma(\overline{\Omega})$. Moreover,
 $$\underline{\kappa}\;\mbox{dist}(x,\partial\Omega)^\gamma\le\Phi_1(x)\le\overline{\kappa}\;\mbox{dist}(x,\partial\Omega)^\gamma\,,$$
 see \cite[Theorem 7.1. and 7.2.]{BFV2018CalcVar} for more details.

 Let us remark that nonnegative functions $u$ belonging to  $H^*(\Omega)$  also belong to $L^1_{\Phi_1}(\Omega)$
 \begin{equation}\label{L1Phi1 to H*}
   \|u\|_{L^1_{\Phi_1}(\Omega)}=\int_\Omega u\,\Phi_1\dx=\int_\Omega \A^{-1/2} u\,\A^{1/2}\Phi_1\le\lambda_1^{1/2}\|u\|_{H^*(\Omega)},
 \end{equation}
 since $u\ge 0$ and $\|\Phi_1\|_{L^2(\Omega)}=1$.

The next theorem  follows the formal proof of the classical  T-contractivity estimate in $L^1_{\Phi_1}$  \cite[Theorem 2.4.]{BII} and the rigorous techniques used to show the $L^1$-contractivity for the linear equation with Caputo derivative \cite[Theorem 7.1.]{WWZ}. Both techniques require extra regularity on the time derivative, more precisely $\partial_t u\in L^1((0,T): L^1_{\Phi_1}(\Omega))$ or $\Dc u\in L^\infty((0,T):L^1(\Omega))$, respectively. Since we only expect, a-priori, that  $\Dc u\in L^2((0,T):H^*(\Omega))$, we first prove T-contractivity  at the discrete level, and  then  comparison principle in the limit $\tau \to 0$. 

 \begin{theorem}\label{Discrete T-contractivity L1Phi1}(Discrete T-contractivity in $L^1_{\Phi_1}(\Omega)$). Let $m>0$ and assume \eqref{K1} and \eqref{K}. Consider $\lbrace U_k\rbrace_{k\ge0},\lbrace V_k\rbrace_{k\ge0}\subset L^{1+m}(\Omega)\cap H^*(\Omega)$  two sequences of solutions for \eqref{DCDP}. Then,
   \begin{equation*}
     \|\left(U_{k}-V_k\right)_+\|_{L^1_{\Phi_1}(\Omega)}\le\|(U_0-V_0)_+\|_{L^1_{\Phi_1}(\Omega)}\qquad\forall k\ge 1\,.
   \end{equation*}
 \end{theorem}

 \begin{proof}
    Notice that $U_k,V_k\in L^{1+m}(\Omega)\subset L^1_{\Phi_1}(\Omega)$ for every $k\ge 0$, since $\Phi_1\in L^\infty(\Omega)$. Therefore, from the definition of discrete Caputo derivative, we obtain
   $$(D^\alpha_\tau U)_k,\;(D^\alpha_\tau V)_k\in L^1_{\Phi_1}(\Omega)\qquad\forall k\ge 1\,.$$
   We will show that $\left(D^\alpha_\tau\|(U-V)_+\|_{L^1_{\Phi_1}(\Omega)}\right)_k\le 0$ for every $k\ge 1$ and then apply Proposition \ref{convex chain rule}.
   For this purpose, let us consider $h_\varepsilon\in C^\infty(\mathbb{R})$ with $\varepsilon>0$ definde by
  $$h_\varepsilon(x)=\sqrt{(x_+)^2+\varepsilon^2}-\varepsilon,\qquad x\in\mathbb{R}\,,$$
  which is a convex and smooth  approximation of $(x)_+$. Moreover, $|h_\varepsilon(x)|\le|x|$ and $|h'_\varepsilon(x)|\le1$ since
  $$h'_\varepsilon(x)=\frac{(x)_+}{\sqrt{x^2+\varepsilon^2}}.$$
  Next, we test equation \eqref{Discrete H* formulation}
  for $U_k$ and $V_k$ with $\A(h'_\varepsilon(U_k^m-V_k^m)\Phi_1)\in H^*(\Omega)$; after subtracting the resulting equations, we get
   \begin{equation}\label{Equation in discrete L1phi1 contractivity}
     \int_\Omega \left(D^\alpha_\tau (U-V)\right)_{k}h'_\varepsilon(U^m_k-V^m_k)\Phi_1\dx=-\int_{\Omega}\A(U^m_k-V^m_k)h'_\varepsilon(U_k^m-V_k^m)\Phi_1\dx\,.
   \end{equation}
   Using  the convexity, the nonnegativity of $h_\varepsilon$, and Kato's inequality \eqref{K}, we  bound the right hand side as
   \begin{align*}
     -\int_{\Omega}\A(U^m_k-V^m_k)\,h'_\varepsilon(U_k^m-V_k^m)\Phi_1\dx&\le-\int_{\Omega}\A\left[h_\varepsilon(U_k^m-V_k^m)\right]\Phi_1\dx \\
     &=-\lambda_1\int_{\Omega}\left(h_\varepsilon (U^m_k-V^m_k)\right)\Phi_1\dx\le 0\,,
   \end{align*}
   since $\Phi_1$ is the first eigenfunction and $\Phi_1\ge 0$. Thus, we deduce from \eqref{Equation in discrete L1phi1 contractivity} and Proposition \ref{convex chain rule} that
   \begin{align*}
      \left(D^\alpha_{\tau}\int_\Omega h_\varepsilon(U_k-V_k)\Phi_1\dx\right)_k&\le\int_\Omega \left(D^\alpha_\tau (U-V)\right)_{k}h'_\varepsilon(U_k-V_k)\Phi_1\dx\\
      &=\int_\Omega \left(D^\alpha_\tau (U-V)\right)_{k}h'_\varepsilon(U^m_k-V^m_k)\Phi_1\dx\\
     &\quad+\int_\Omega (D^\alpha_\tau (U-V))_{k}\left(h'_\varepsilon(U_k-V_k)-h'_\varepsilon(U^m_k-V^m_k)\right)\Phi_1\dx\\
     &\le I_\varepsilon\,,
   \end{align*}
   where
   $$I_\varepsilon:=\int_\Omega \left|(D^\alpha_\tau (U-V))_{k}\right|\left|h'_\varepsilon(U_k-V_k)-h'_\varepsilon(U^m_k-V^m_k)\right|\Phi_1\dx\,.$$
   To prove that $I_\varepsilon\rightarrow 0$ when $\varepsilon\rightarrow0$, we use dominated convergence theorem: note that $(D^\alpha_\tau(U-V))_k\in L^1_{\Phi_1}(\Omega)$ for every $k\ge 1$, $|h_\varepsilon'(x)-h_\varepsilon'(y)|\le 2$ for every $x,y\in\mathbb{R}$, and $h'_\varepsilon\rightarrow h'$ pointwise. On the other hand $|h_\varepsilon(x)|\le|x|$, and therefore
   $$\left|\left(D^\alpha_\tau h_\varepsilon(U-V)\right)_k\right|\le\sum_{j=0}^{k}\lambda_j\left(|U_j|+|V_j|\right)\in L^1_{\Phi_1}(\Omega)\,$$
   for some $\lambda_j>0$. Hence, the limit $\varepsilon\rightarrow 0$ yields
   $$\left(D^\alpha_{\tau}\int_\Omega h_\varepsilon(U_k-V_k)\Phi_1\dx\right)_k\rightarrow\left(D^\alpha_{\tau}\int_\Omega (U_k-V_k)_+\Phi_1\dx\right)_k\qquad\forall k\ge1\,,$$
   since  $h_\varepsilon(x)\rightarrow(x)_+$. Therefore,
   $$\left(D^\alpha_{\tau}\int_\Omega (U_k-V_k)_+\Phi_1\dx\right)_k\le 0\qquad\forall k\ge 1\,,$$
    and the result follows by Proposition \ref{convex chain rule}.
 \end{proof}

Using the theorem above we prove the following discrete comparison principle:

 \begin{corollary}\label{Discrete Comparison Principle}(Discrete Comparison Principle).
    Let $m>0$ and assume \eqref{K1} and \eqref{K}. Consider  two sequence of solutions of \eqref{DCDP}, namely, $\lbrace U_k\rbrace_{k\ge0},\lbrace V_k\rbrace_{k\ge0}\subset L^{1+m}(\Omega)\cap H^*(\Omega)$. Then, if $U_0(x)\le V_0(x)$ for a.e. $x\in\Omega$, it holds that for every $k\ge 1$
    \begin{equation*}
      U_k(x)\le V_k(x)\qquad\mbox{for a.e. }x\in\Omega\,.
    \end{equation*}
    Moreover, if $V_0\ge 0$ a.e. in $\Omega$ then $V_k\ge 0$ a.e. in $\Omega$.
 \end{corollary}
 \begin{proof}
   By Theorem \ref{Discrete T-contractivity L1Phi1} we know that for every $k\ge 1$
   $$\|(U_k-V_k)_+\|_{L^1_{\Phi1}(\Omega)}\le\|(U_0-V_0)_+\|_{L^1_{\Phi1}(\Omega)}\,.$$
   Since $U_0\le V_0$ a.e. in $\Omega$, we have that $\|(U_0-V_0)_+\|_{L^1_{\Phi1}(\Omega)}=0$. Hence,
   $$\|(U_k-V_k)_+\|_{L^1_{\Phi_1}(\Omega)}=0\,\qquad\mbox{ for all }k\ge 1\,.$$
    The comparison principle follows since $\Phi_1(x)\ge\underline{\kappa}\;\mbox{dist}(x,\partial\Omega)$. For the nonnegativity of $V_k$ we just take $U_k=0$ for every $k\ge 0$, which is the trivial solution.
 \end{proof}

\section{Continuous problem}\label{section: Continuous problem}
In this section we prove Theorem \ref{Comparison Principle}. For this purpose, we recall the result in \cite{LiLiu2019} about the convergence of the discrete solutions of \eqref{Discrete Gradient Flow} to $H^*$-solutions of \eqref{CDP}. In addition, we recollect some other results about existence, uniqueness and estimates for solutions exhibited in \cite{GA}. At the end of this section, we provide a proof of the comparison principle for the $H^*$-solutions by using the discrete comparison principle of Theorem \ref{Discrete Comparison Principle} and the convergence of the discrete scheme. The statements in \cite{GA} and \cite{LiLiu2019} are in the general framework of subdifferentials of $\lambda$-convex energy functionals in Hilbert spaces; here we adjust  these statements for the Porous Medium equation.

\subsection{Existence and uniqueness}

In \cite{LiLiu2019}, the authors show that solutions of the implicit Euler scheme \eqref{Discrete Gradient Flow} converge with certain rates to the unique $H^*$-solution, following the steps of the classical theory of gradient flow of Brezis \cite{Brezis-Book-semigr} and Komura \cite{K67}. The theorem below will be crucial in the proof of the comparison principle.

\begin{theorem}\cite[Theorem 5.10]{LiLiu2019}\label{existence of limit solutions}
Let $m>0$ and assume \eqref{K1}. Then, for every initial datum $u_0\in L^{1+m}(\Omega)\cap H^{-1}(\Omega)$ there exits a unique $H^*$ solution $u=u(t,x)$ of \eqref{CDP}. Moreover, it satisfies that
\begin{align}\label{Holder continuity of limit solution}
    \|u(t+\tau)-u(t)\|_{H^*(\Omega)}\le C|\tau|^{\alpha/2}\qquad\forall\tau>0\,,
\end{align}
and the following estimates for the error with respect to the discrete solution \eqref{Minimizing movement scheme} holds,
\begin{equation}\label{Convergence of the discretization}
  \sup\limits_{k}\|U_k-u(t_k)\|_{H^*(\Omega)}\le C(u_0,T)\;\tau^{\alpha/4}\qquad\forall\tau>0\,.
\end{equation}
\end{theorem}

Convergence of the discrete scheme in a more general setting has been proven very recently by  Li and Salgado in. \cite[Theorem 4.5.]{LiSalgado}.

A different proof of existence and uniqueness of $H^*$-solution of \eqref{CDP} is provided by Akagi  in \cite{GA}.  Akagi  approximates the problem by adding a classical time derivative $\varepsilon\,\partial_t u_\varepsilon$. After proving the existence for the equation $\varepsilon\partial_t u_\varepsilon+\Dc u_\varepsilon=-\A u_\varepsilon^m$ using the theory of maximal monotone operators, he passes to the limit $\varepsilon\rightarrow0$ thanks to uniform estimates. We   will use some of these estimates in the proof of smoothing effects. Below we recall the most important ones:
\begin{theorem}\cite[Theorem 2.3. and 6.2]{GA}\label{Existence Akagi GF}
  Let $m>0$ and assume \eqref{K1}. Let $u_0\in L^{1+m}(\Omega)\cap H^{-1}(\Omega)$. Then, there exits a unique $H^*$ solution $u=u(t,x)$ of \eqref{CDP}. This solution $u$ satisfies:
  \begin{equation}\label{Akagi estimate}
  \begin{split}
    u\in L^\infty((0,T):H^*(\Omega)),\quad |u|^{m-1}u\in L^2((0,T):H(\Omega))\,,\\
    \ell*\left\| \Dc u\right\|_{H^*(\Omega)}^2\in L^\infty((0,T))\qquad\qquad\,,
  \end{split}
  \end{equation}
  and
  \begin{align}\label{Control of the energy}
    \Es(u(t))\le\Es(u_0)\qquad t>0\,.
  \end{align}
  Moreover,  contractivity in $H^*(\Omega)$ holds, that is, if $u$ and $v$ are $H^*$ solutions with initial datum $u_0,v_0\in L^{1+m}(\Omega)\cap H^{-1}(\Omega)$, respectively, then
  \begin{equation}
    \|u(t)-v(t)\|_{H^*(\Omega)}\le\|u_0-v_0\|_{H^*(\Omega)}\qquad\;\mbox{for a.e. }\;t\in(0,T)\,.
  \end{equation}
\end{theorem}
  Since there is uniqueness of $H^*$-solution, Theorem \ref{existence of limit solutions} and Theorem \ref{Existence Akagi GF} describe the same solution for a fixed $u_0\in L^{1+m}(\Omega)\cap H^*(\Omega)$. Despite the difference in the approximation procedures, the limit solution coincides. Hence, we can use the estimates from both theorems simultaneously.

\subsection{Comparison Principle}\label{Ssec.Comparison}
Now, we are ready to prove the comparison principle for the $H^*$-solutions of \eqref{CDP}. Let us recall the precise statement: 
\begin{theorem}\label{comparison in section}
  Let $m>0$ and assume  \eqref{K1} and \eqref{K}. Let $u$ and $v$ be two $H^*$-solutions of \eqref{CDP} with initial data $u_0,v_0\in L^{1+m}(\Omega)\cap H^*(\Omega)$, respectively. If $u_0\le v_0$ a.e. in $\Omega$, then $\forall t>0$
  \begin{equation*}
    u(t)\le v(t)\qquad\mbox{ a.e. in }\Omega\,.
  \end{equation*}
\end{theorem}
\noindent{\it Proof.} The idea  is to inherit the comparison principle from the discrete comparison principle shown in Corollary \ref{Discrete Comparison Principle}. Let us consider two solutions $u,v$ of \eqref{CDP} with initial data $u_0,v_0\in L^{1+m}(\Omega)\cap H^*(\Omega)$ respectively, satisfying
\begin{equation}\label{Comparison assumption}
  u_0(x)\le v_0(x)\qquad\mbox{ for a.e. }x\in\Omega\,.
\end{equation}
 We proceed by contradiction. Assume that there exist $t_0>0$ and $\widetilde\Omega\subset\Omega$ with $|\widetilde{\Omega}|>0$ such that
\begin{equation}\label{contradiction hyp}
  u(t_0,x)-v(t_0,x)>0\qquad\mbox{ for a.e. }x\in\widetilde\Omega\,.
\end{equation}
Then, we can define the sequence of time indexes $\lbrace n_\tau\rbrace_{\tau>0}$ such that for all $\tau>0$ we have $t_{n_\tau}=\tau \,n_\tau$ and
  $$t_0\in[t_{n_\tau},t_{n_\tau+1})\,.$$
Let us consider the respective discrete approximations of $u$ and $v$ as in \eqref{Minimizing movement scheme}, namely, the sequences $\lbrace U_k\rbrace_{k\ge0},\,\lbrace V_k\rbrace_{k\ge 0}\subset L^{1+m}(\Omega)\cap H^*(\Omega)$ with $U_0=u_0$ and $V_0=v_0$. Notice that from \eqref{Comparison assumption} and the discrete comparison principle in Theorem \ref{Discrete Comparison Principle}, it holds that for every $k\ge0$
$$U_k(x)-V_k(x)=-\left(U_k(x)-V_k(x)\right)_-\qquad\mbox{ for a.e. }x\in\Omega\,.$$
Hence, we deduce from \eqref{Convergence of the discretization} that
 \begin{align}\label{H* convergence for comparison}
   \|(u(t_{n_\tau})-v(t_{n_\tau}))+\left(U_{n_\tau}-V_{n_\tau}\right)_-\|_{H^*(\Omega)}\rightarrow 0\qquad\mbox{as }\tau\rightarrow0\,.
 \end{align}
 Note that by \eqref{Holder continuity of limit solution}, we also know that
 $$\|(u(t_{n_\tau})-v(t_{n_\tau}))-(u(t_0)-v(t_0))\|_{H^*(\Omega)}\rightarrow 0 \qquad\mbox{as }\tau\rightarrow0\,,$$
 therefore each of the terms $u(t_{n_\tau})-v(t_{n_\tau})$ and $ U_{n_\tau}-V_{n_\tau}$  in \eqref{H* convergence for comparison} converge (strongly in $H^*(\Omega)$) to $u(t_0)-v(t_0)\in L^{1+m}(\Omega)$.

In addition, the control of the energy in \eqref{discrete Decay of the Energy} and \eqref{Control of the energy} implies the following uniform bound of the $L^{1+m}$-norm
 \begin{equation*}
   \|(u(t_{n_\tau})-v(t_{n_\tau}))+\left(U_{n_\tau}-V_{n_\tau}\right)_-\|_{L^{1+m}}\le 2\left(\|u_0\|_{L^{1+m}}+\|v_0\|_{L^{1+m}}\right)\qquad\forall \tau>0\,.
 \end{equation*}
 Thus, the sequence $\left\lbrace (u(t_{n_\tau})-v(t_{n_\tau}))+\left(U_{n_\tau}-V_{n_\tau}\right)_-\right\rbrace_{\tau>0}$ is weakly compact in $L^{1+m}(\Omega)$ and by \eqref{H* convergence for comparison} and uniqueness of the weak-strong limit, we obtain that up to a subsequence
 \begin{equation}
   (u(t_{n_\tau})-v(t_{n_\tau}))+\left(U_{n_\tau}-V_{n_\tau}\right)_-\rightharpoonup 0\qquad\mbox{ in }L^{1+m}(\Omega)\,,
 \end{equation}
 or equivalently,
 \begin{equation}\label{weak convergenece in Lp}
   \int_\Omega\left[(u(t_{n_\tau})-v(t_{n_\tau}))+\left(U_{n_\tau}-V_{n_\tau}\right)_-\right]\psi\dx\rightarrow0\qquad\forall \psi\in L^{\frac{1+m}{m}}(\Omega)\,.
 \end{equation}
We proceed similarly  for the sequence $\left\lbrace (u(t_{n_\tau})-v(t_{n_\tau}))-(u(t_0)-v(t_0))\right\rbrace_{\tau>0}$ and thanks to
\begin{equation*}
\|(u(t_{n_\tau})-v(t_{n_\tau}))-(u(t_0)-v(t_0))\|_{1+m}\le 2\left(\|u_0\|_{1+m}+\|v_0\|_{1+m}\right)\qquad\forall\tau>0\,,
\end{equation*}
we obtain
\begin{equation}\label{2weak convergenece in Lp}
   \int_\Omega\left[(u(t_{n_\tau})-v(t_{n_\tau}))-\left(u(t_0)-v(t_0)\right)\right]\psi\dx\rightarrow0\qquad\forall \psi\in L^{\frac{1+m}{m}}(\Omega)\,.
 \end{equation}
Let us choose $\widetilde\psi(x):=\chi_{\widetilde\Omega}(x)\in  L^{\infty}(\Omega)$ to be the characteristic function of $\widetilde{\Omega}$, in order to have
 \begin{align*}
    A_\tau:=&\int_\Omega\left[(u(t_{n_\tau})-v(t_{n_\tau}))+(U_{n_\tau}-V_{n_\tau})_-\right]\widetilde\psi\dx\\
    \ge&\int_{\Omega}\left(u(t_{n_\tau})-v(t_{n_\tau}))-(u(t_0)-v(t_0)\right)\widetilde\psi\dx +\int_{\Omega}\left(u(t_0)-v(t_0)\right)\widetilde\psi\dx\\
    =&\underbrace{\int_{\Omega}\left(u(t_{n_\tau})-v(t_{n_\tau}))-(u(t_0)-v(t_0)\right)\widetilde\psi\dx }\limits_{B_\tau\rightarrow 0} +\underbrace{\int_{\widetilde\Omega}\left(u(t_0)-v(t_0)\right)\dx}\limits_{C>0}\,.
  \end{align*}
Finally, we obtain a contradiction, since  \eqref{weak convergenece in Lp} and \eqref{2weak convergenece in Lp} implies $A_\tau\rightarrow 0$ and $B_\tau\rightarrow 0$ when $\tau\rightarrow0$, but from the contradiction hypothesis \eqref{contradiction hyp} we know that $C>0$ and it does not depend on $\tau$.

 \hfill$\square$

\section{Smoothing effects}\label{section: Smoothing}

The aim of this section is to prove the $L^p-L^\infty$ smoothing effects stated in Theorem \ref{Thm smoothing effect}. We divide the proof in several steps. First, in Proposition \ref{Prop: Fundamental pointwise formula} we use the dual formulation of the problem and the Fundamental Theorem of Calculus for Caputo derivative to obtain a fundamental pointwise formula. Then, in Theorem \ref{Benilan-Crandall}  we obtain a time monotonicity formulae for the function $t^{\frac{\alpha}{m-1}}u(t)$. The time monotonicity formulae we obtain are consistent with the case $\alpha=1$, since they depend only on the time scaling and the comparison principle. Next, we apply the time monotonicity formulae in the fundamental pointwise formula in order to obtain the upper bounds of $u^m(t,x)$, as described in Proposition \ref{Fundamental pointwise estimate}. We conclude the proof of Theorem \ref{Thm smoothing effect} and Theorem \ref{thm: Upper boundary estimates 2} using assumptions \eqref{K1} and \eqref{K2}, respectively, and Hölder's inequality.

Before starting with the first proposition of this section, let us present a useful lemma about the differentiability of the time convolution.

\begin{lemma}\label{Regularity of convolutions}
  Let $f\in L^1((0,T))$ and $g\in W^{1,2}((0,T))$ with $g(0)=0$, then $f*g\in W^{1,2}((0,T))$ and $$\frac{\rm d}{\dt}(f*g)=\left(f*\frac{{\rm d}g}{\dt}\right)\,.$$
\end{lemma}
\begin{proof}
We obtain that $f*g\in L^2((0,T))$ by Young's inequality.  For any $\varphi\in C_c^{\infty}([0,T])$ let us consider
  \begin{align}\label{derivative of convolution}
    \int_{0}^{T}(f*g)(t)\varphi'(t)\dt&=\int_{0}^{T}\int_{0}^{t}f(t-\tau)g(\tau)\varphi'(\tau)\dtau\dt =\int_{0}^{T}\left(\int_{\tau}^{T}f(t-\tau)\varphi'(\tau)\dt\right)g(\tau)\dtau.
  \end{align}
  Now, we show that for a.e. $\tau\in(0,T)$
  $$\int_{\tau}^{T}f(t-\tau)\varphi'(\tau)\dt=\int_{0}^{T-\tau}f(t)\varphi'(\tau+t)\dt=\left(\int_{0}^{T-\tau}f(t)\varphi(\tau+t)\dt\right)'=\left(\int_{\tau}^{T}f(t-\tau)\varphi(\tau)\dt\right)'.$$
  For this purpose, let us study the incremental quotient
  \begin{align*}
    &\frac{1}{h}\left[\int_0^{T-\tau-h}f(t)\varphi(t+\tau+h)\dt-\int_0^{T-\tau}f(t)\varphi(t+\tau)\dt\right]\\
    =& \int_{0}^{T-\tau}f(t)\chi_{[0,T-\tau-h]}(t)\frac{\varphi(t+\tau+h)-\varphi(t+\tau)}{h}\dt-\frac{1}{h}\int_{T-\tau-h}^{T-\tau}f(t)\varphi(t+\tau)\dt.
   \end{align*}
  The first term converge by dominated convergence theorem when $h$ goes to 0. For the second term we use Lebesgue differentiation theorem  and the fact that $\varphi(T)=0$ to prove that it converge to 0 as $h$ goes to 0. Hence, we conclude from \eqref{derivative of convolution} and integration by parts that
  \begin{align*}
    \int_{0}^{T}(f*g)(t)\varphi'(t)\dt&=\int_{0}^{T}\left(\int_{\tau}^{T}f(t-\tau)\varphi(t)\right)'g(\tau)\dtau=-\int_{0}^{T}\int_{\tau}^{T}f(t-\tau)\varphi(t)g'(\tau)\dtau\\
    &=\int_{0}^{T}(f*g')(t)\varphi(t)\dt,
  \end{align*}
  since $g(0)=0$.
\end{proof}
Taking in to account estimate \eqref{Akagi estimate}  from the Existence Theorem \ref{Existence Akagi GF}, we make the dual formulation of the problem rigorous and we prove the Fundamental Theorem of Calculus for Caputo in the following proposition.
\begin{proposition}[Fundamental pointwise formula]\label{Prop: Fundamental pointwise formula}
  Let $u$ be a $H^*$-solution of \eqref{CDP} with
  \begin{equation}\label{caputo FTC regularity}
    \ell*\|\Dc u\|^2_{H^*(\Omega)}\in L^\infty((0,T))\,.
  \end{equation}
  Then, for a.e. $t\in(0,T)$ and a.e. $x_0\in\Omega$ it holds that
  \begin{equation}\label{Fundamental integral formula}
    \int_\Omega(u_0(x)-u(t,x))\G(x_0,x)\dx=\frac{1}{\Gamma(\alpha)}\int_{0}^{t}\frac{u^m(\tau,x_0)}{(t-\tau)^{1-\alpha}}\dtau.
  \end{equation}
\end{proposition}
\begin{proof}
The intuition behind the proof is to formally use the following test function in \eqref{strong H solution},
$$\psi(\tau,x):=\frac{\chi_{(0,t)}(\tau)}{(t-\tau)^{1-\alpha}}\;\delta_{x_0}(x),$$
and integrate on $[0,T]$. Formally, on the left hand side we would have
\begin{align*}
  \int_{0}^{T}\int_\Omega\Dc u(\tau,x)\AI\psi(\tau,x)\dx\dtau &=\ell*\left(\int_\Omega\Dc u\AI \delta_{x_0}\dx\right)=\ell*k*\left(\frac{\rm d}{\dt}\int_\Omega (u(t,x)\G(x_0,x)\dx\right)\\
  &=\int_\Omega(u(t,x)-u_0(x))\G(x_0,x)\dx\,,
\end{align*}
and on the right hand side
\[
-\int_{0}^{T}\int_\Omega u^m(\tau,x)\psi(\tau,x)\dx\dtau=-\int_{0}^{t}\frac{u^m(\tau,x_0)}{(t-\tau)^{1-\alpha}}\dtau\,.
\]
In order to give a rigorous proof, we split the proof in two steps.\vspace{2mm}

\noindent$\bullet~$\textbf{Step 1.} First, we show that for any $\varphi\in H^*(\Omega)$ we have
\begin{equation}\label{Step 1 applying FTC}
\int_\Omega(u(t,x)-u_0(x))\AI\varphi(x)\dx=-\frac{1}{\Gamma(\alpha)}\int_{0}^{t}\int_\Omega\frac{u^m(\tau,x)}{(t-\tau)^{1-\alpha}}\varphi(x)\dx\dtau\,,
\end{equation}
by testing the equation with
$$\tilde\psi(\tau,x)=\frac{\chi_{(0,t)}(\tau)}{(t-\tau)^{1-\alpha}}\,\varphi(x)\in L^1(0,T;H^*(\Omega))\,.$$
Note that condition \eqref{caputo FTC regularity} implies that $\tilde\psi$ is an admissible test function, and $\ell*\left(\Dc u\AI\varphi\dx\right)\in L^\infty(0,T)$, since
  \begin{align*}
    \left|\int_{0}^{t}\int_\Omega\frac{\Dc u(\tau,x)\AI \varphi(x)}{(t-\tau)^{1-\alpha}}\dx\dtau\right| &\le\|\varphi\|_{H^*(\Omega)}\int_{0}^{t}\frac{\|\Dc u(\tau)\|_{H^*(\Omega)}}{(t-\tau)^{1-\alpha}}\dtau\\
    &\le\|\varphi\|_{H^*(\Omega)}\left(\int_{0}^{t}\frac{\|\Dc u(\tau)\|^2_{H^*(\Omega)}}{(t-\tau)^{1-\alpha}}\dtau\right)\left(\int_{0}^{t}\frac{\dtau}{(t-\tau)^{1-\alpha}}\right)^{1/2}\\
    &\le t^{\alpha/2}\|\varphi\|_{H^*(\Omega)}\;\left(\ell*\|\Dc u\|^2_{H^*(\Omega)}(t)\right)^{1/2}<+\infty\,.
  \end{align*}
  Hence, applying Fubini's Theorem and Lemma \ref{Regularity of convolutions} we obtain
  \begin{align*}
    \ell*\left(\int_\Omega\Dc u\AI\varphi\dx\right)\!(t)\!&=\!\!\int_\Omega \!\!\ell*\!\left(\frac{\rm d}{\dt}[k*(u-u_0)]\right)\!\!(t)\AI\varphi\dx
    =\!\int_\Omega \!\!\frac{\rm d}{\dt}\left(\ell*[k*(u-u_0)](t)\right)\AI\varphi\dx\\
    &=\int_\Omega \frac{\rm d}{\dt}(1*(u-u_0)(t))\AI\varphi\dx=\int_\Omega (u(t)-u_0)\AI\varphi\dx\,.
  \end{align*}
  \noindent$\bullet~$\textbf{Step 2.} Now, we consider the sequence $\varphi_r(x)=\frac{\chi_{B_{r}(x_0)}(x)}{|B_{r}|}\in H^*(\Omega)$ with $x_0\in\Omega$ being a Lebesgue point and we pass to the limit in \eqref{Step 1 applying FTC} when $r\rightarrow0$.
  \end{proof}

  Next, we want  to bound from below the integral on the right hand side of \eqref{Fundamental integral formula}, by some pointwise value of the solution of $u(t,x_0)$. Therefore, we prove the two time monotonicity formulae of Theorem \ref{Thm Benilan-Crandall}, depending on whether $0<m<1$ or $m>1$. This type of time monotonicity estimates are the analogous of the Benilan-Crandall estimates for the classical problem. In the context of Caputo derivatives, time monotonicity formulae are a surprising consequence of the comparison principle and the time scaling of \eqref{CDP}.

  \begin{theorem}[Time monotonicity formulae]\label{Benilan-Crandall}
  Let $u$ be a nonnegative $H^*$-solution for \eqref{CDP} under assumptions of Theorem \ref{Comparison Principle}.
  \begin{enumerate}
    \item If $0<m<1$, then the function $$t\mapsto t^{\frac{-\alpha}{1-m}}u(t,x)$$ is non increasing for any $t>0$ and a.e $x\in\Omega$.
    \item If $m>1$, then the function $$t\mapsto t^{\frac{\alpha}{m-1}}u(t,x)$$ is non decreasing for  any $t>0$ and a.e $x\in\Omega$
  \end{enumerate}
\end{theorem}
\begin{proof}
  The proof follows by time scaling and the comparison principle of Theorem \ref{Comparison Principle}. For any $\lambda>0$, we define the rescaled solution $$u_\lambda(t,x):=\lambda^{-\frac{\alpha}{1-m}}u(\lambda t,x)$$ which is a solution to \eqref{CDP} since
   $$\Dc\left(\lambda^{-\frac{\alpha}{1-m}}u(\lambda t)\right)=\lambda^{\alpha-\frac{\alpha}{1-m}}\left(\Dc u\right)(\lambda t)=-\lambda^{-\frac{\alpha m}{1-m}}\A u^m(\lambda t)=-\A\left(\lambda^{-\frac{\alpha}{1-m}}u(\lambda t)\right)^m\,.$$
   \noindent {\it Case $0<m<1$:} Let us choose $0<t_0<t_1$ and fix $\lambda=\frac{t_1}{t_0}\ge 1$, then it holds that
  \begin{align*}
    u_\lambda(0,x)=\left(\frac{t_1}{t_0}\right)^{-\frac{\alpha}{1-m}}u_0(x)\le u_0(x)\qquad\mbox{a.e.}\quad x\in \Omega\,.
  \end{align*}
  By comparison principle of Theorem \ref{Comparison Principle}, we have that $u_\lambda(t,x)\le u(t,x)$ for any $t>0$ and a.e. $x\in\Omega$. Hence, we can take $t=t_0$ in order to conclude that
  $$\left(\frac{t_1}{t_0}\right)^{-\frac{\alpha}{1-m}}u(t_1,x)\le u(t_0,x)\qquad\mbox{for a.e.}\quad x\in\Omega\,.$$
  \noindent {\it Case $m>1$:} Let us consider any $0<t_0<t_1$ and fix $\lambda=\frac{t_0}{t_1}\le 1$, then
  $$u_\lambda(0,x)=\left(\frac{t_0}{t_1}\right)^{\frac{\alpha}{m-1}}u_0(x)\le u_0(x)\qquad\mbox{a.e.}\quad x\in \Omega\,.$$
   As in the previous case, we use comparison principle of Theorem \ref{Comparison Principle} in order to obtain $u_\lambda(t,x)\le u(t,x)$ for any $t>0$ and a.e. $x\in\Omega$. Therefore, we choose $t=t_1$ to get
  $$\left(\frac{t_0}{t_1}\right)^{\frac{\alpha}{m-1}}u(t_0,x)\le u(t_1,x)\qquad\mbox{for a.e.}\quad x\in\Omega\,.$$
\end{proof}

Afterwards, we implement the time monotonicity formulae of Theorem \ref{Benilan-Crandall} in the equality \eqref{Fundamental integral formula} in order to obtain the following fundamental upper bounds.

\begin{proposition}[Fundamental upper bounds]\label{Fundamental pointwise estimate}
  Assume \eqref{K1} and \eqref{K} and let $u$ be a bounded and nonnegative $H^*$-solution for \eqref{CDP}. Then for a.e. $x_0\in\Omega$ and a.e. $t\in(0,T)$ it holds that
  \begin{enumerate}
    \item If $0<m<1$, then
    \begin{equation}
    u^{m}(t,x_0)\le \frac{\c}{t^\alpha}\int_\Omega(u_0(x)-u(t,x))\G(x,x_0)\dx,
  \end{equation}
  with $\c=\c(m,\alpha)>0$.
    \item If $m>1$, then
    \begin{equation}
    u^{m}(t,x_0)\le \frac{\c}{t^\alpha}\int_\Omega(u_0(x)-u(2t,x))\G(x,x_0)\dx,
  \end{equation}
  with $\c=\c(m,\alpha)>0$.
  \end{enumerate}
\end{proposition}
\begin{proof}
The idea of the proof is to use Benilan-Crandall estimates \eqref{Benilan-Crandall} on the right hand side of the fundamental integral formula \eqref{Fundamental integral formula}. Namely, it suffices to prove the following inequality
\begin{equation}\label{pointwise-integral ineq}
  \c \;t^\alpha\, u^{m}(t,x_0)\le\int_0^{t_1}\frac{u^{m}(\tau,x_0)}{(t_1-\tau)^\alpha}\dtau,
\end{equation}
with $\c=\c(m,\alpha)>0$ and $t_1=t$ if $0<m<1$ and $t_1=2t$ for $m>1$.\\
\noindent$\bullet~${\it Case $0<m<1$:}  By Theorem \ref{Benilan-Crandall}, we have that
 $$u^{m}(\tau,x_0)\ge\left( \frac{\tau}{t}\right)^{\frac{\alpha m}{1-m}}u^{m}(t,x_0)\qquad \mbox{for}\quad 0<\tau<t\,.$$
 Then,
 \[
\begin{split}
  \int_{0}^{t}\frac{u^{m}(\tau,x_0)}{(t-\tau)^{1-\alpha}}\dtau  &\ge \frac{u^{m}(t,x_0)}{t^{\frac{\alpha m}{1-m}}}\int_{0}^{t}\frac{\tau^{\frac{\alpha m}{1-m}}}{(t-\tau)^{1-\alpha}}\dtau\\
  &= t^\alpha u^{m}(t,x_0)\int_{0}^{1}\frac{r^{\frac{\alpha m}{1-m}}}{(1-r)^{1-\alpha}}{\rm d}r\\
  &=B\left(\frac{\alpha m}{1-m}+1,\alpha\right) t^{\alpha} u^{m}(t,x_0)\,,
\end{split}
\]
where we used the change of variable $tr=\tau$ and the definition of the beta function $B(z_1,z_0)=\int_{0}^{1}\tau^{z_1-1}(1-\tau)^{z_2-1}{\rm d}\tau$, which is well finite for $z_1,z_2>0$.\vspace{3mm}

\noindent$\bullet~${\it Case $m>1$:}  By Theorem \ref{Benilan-Crandall}, we have that
$$u^{m}(\tau)\ge \left(\frac{t}{\tau}\right)^{\frac{\alpha m}{m-1}} u^{m}(t,x_0) \qquad\mbox{for}\quad 0<t<\tau<2t\,.$$
Then,
\begin{align*}
  \int_{0}^{2t}\frac{u^{m}(\tau,x_0)}{(2t-\tau)^{1-\alpha}}\dtau&\ge\int_{t}^{2t}\frac{u^{m}(\tau,x_0)}{(2t-\tau)^{1-\alpha}}\dtau\\
  &\ge t^{\frac{\alpha(m)}{m-1}}u^{m}(t,x_0)\int_{t}^{2t}\frac{\dtau}{(2t-\tau)^{1-\alpha}\;\tau^{\frac{\alpha m}{m-1}}}\\
  &\ge\frac{1}{2^{\frac{\alpha m}{m-1}}}\; u^{m}(t,x_0)\int_{t}^{2t}\frac{\dtau}{(2t-\tau)^{1-\alpha}}\\
  &=\frac{1}{\alpha\,2^{\frac{\alpha m}{m-1}}}\;t^\alpha \,u^{m}(t,x_0)\,.
\end{align*}
The proof follows by combining \eqref{Fundamental integral formula} and \eqref{pointwise-integral ineq}.
\end{proof}

Next, we recall a result from \cite{BFV2018CalcVar} involving $L^q$ estimates of the Green function.

\begin{lemma}[Green function estimates, \cite{BFV2018CalcVar}]\label{Lem.Green}Let $\G$ be the kernel of $\AI$, and assume that \eqref{K1} holds. Then, for all $0<q<{N}/(N-2s)$, there exist a constant $c_{2,\Omega}(q)>0$ such that
	\begin{equation}\label{Lem.Green.est.Upper.I}
	\sup_{x_0\in\Omega}\int_{\Omega}\G^q(x , x_0)\dx \le c_{2,\Omega}(q)\,.
	\end{equation}
 The constant $c_{2,\Omega}$ depends only on $s,N, q, \Omega$, and have an explicit expression, cf. \cite{BFV2018CalcVar}.
\end{lemma}

Finally, we are in the position to prove the $L^p-L^\infty$ smoothing effects of Theorem \ref{Thm smoothing effect}, namely 
$$\|u(t)\|_\infty\le\ka\frac{\|u_0\|^{\frac{1}{m}}_p}{t^{\frac{\alpha}{m}}}\qquad\forall t>0\,.$$

\noindent{\it Proof of Theorem \ref{Thm smoothing effect}.} First, we prove the estimate for nonnegative solutions and then generalize the result for sign-changing solutions.

 \noindent$\bullet~$\textbf{Step 1. }{\it(Smoothing effect for nonnegative solutions).} If we combine Proposition \ref{Fundamental pointwise estimate} and Lemma \ref{Lem.Green}, the boundedness of solutions $u(t)$ follows for any $t>0$ given that $0\le u_0\in L^{q'}(\Omega)$ with $q'>\frac{N}{2s}$:
  \begin{align*}
    \|u(t)\|_\infty^m&\le \sup\limits_{x_0\in\Omega}\frac{\c}{t^\alpha}\int_\Omega(u_0(x)-u(t_1,x))\G(x,x_0)\dx\le\sup\limits_{x_0\in\Omega}\frac{\c}{t^\alpha}\int_\Omega u_0(x)\G(x,x_0)\dx\\
    &\le\frac{\c}{t^\alpha}\|u_0\|_{q'}\sup\limits_{x_0\in\Omega}\|\G(\cdot,x_0)\|_{q}\le C \frac{\|u_0\|_{q'}}{t^\alpha},
  \end{align*}
  where we have used the nonnegativity of $u(t_1,x)$ with $t_1=t$ if $0<m<1$ and $t_1=2t$ if $m>1$, and Young's inequality. The constant $C$ depends on $\alpha,m,s,N$ and $\Omega$.

 \noindent$\bullet~$\textbf{Step 2. }{\it(Smoothing effect for sign-changing solutions).} We reduce the problem to nonnegative initial data using the comparison principle. We split the signed initial datum,
 $$u_0=(u_0)_+-(u_0)_-\,,$$
 and consider the nonnegative $H^*$-solutions $\tilde{u}^+$ and $\tilde{u}^-$ with nonnegative initial data $(u_0)_+,(u_0)_-$, respectively. First, we have to check that $(u_0)_+,(u_0)_-\in L^p(\Omega)\cap H^*(\Omega)$ with $p$ as in \eqref{Condition for p}. Since $u_0\in L^p(\Omega)\cap H^*(\Omega)$ with $p$ satisfying \eqref{Condition for p}, then $(u_0)_+,(u_0)_-\in L^p(\Omega)$ with the same $p$. It remains to prove that they also belong to $H^*(\Omega)$. For this purpose, we use the Hardy-Littlewood-Sobolev inequality associated to the quadratic form of the operator: There exists $\mathcal{H}_{\A}>0$ such that
 \begin{equation}
   \|f\|_{H^*(\Omega)}\le\mathcal{H}_\A\|f\|_{\frac{2N}{N+2s}}\,,
 \end{equation}
 which holds under assumption \eqref{K1}, see \cite[Theorem 7.5]{BSV2013}. Since $N>2s$ then $\frac{2N}{N+2s}<\frac{N}{2s}$, and hence
 $$\|(u_0)_\pm\|_{H^*(\Omega)}\le\mathcal{H}_\A\|(u_0)_\pm\|_{\frac{2N}{N+2s}}\le\mathcal{H}_\A\left|\Omega\right|^{\frac{N-2s}{2N}}\|(u_0)_\pm\|_{\frac{N}{2s}}\,.$$
 Now, we use the smoothing effects for nonnegative initial data of Step 1 to obtain
\begin{equation*}
  \|\tilde{u}^\pm(t)\|_\infty\le C\,\frac{\|(u_0)_\pm\|_p^{\frac{1}{m}}}{t^{\frac{\alpha}{m}}}\qquad\forall t>0\,.
\end{equation*}
 By the comparison principle of Theorem \ref{Comparison Principle}, $u_0\le(u_0)_+$ a.e. in $\Omega$, therefore we have for all $t>0$
 $$u(t,x)\le\tilde{u}^+(t,x)\le C\,\frac{\|(u_0)_+\|_p^{\frac{1}{m}}}{t^{\frac{\alpha}{m}}}\qquad\mbox{ for  a.e. }x\in\Omega\,.$$
 For the lower bound, note that $-u_0\le(u_0)_-$ a.e. in $\Omega$. Hence, by comparison principle it holds that for every $t>0$
 $$-u(t,x)\le\tilde{u}^-(t,x)\le C\,\frac{\|(u_0)_-\|_p^{\frac{1}{m}}}{t^{\frac{\alpha}{m}}}\qquad\mbox{ for a.e. }x\in \Omega\,,$$
 and the result follows.

\hfill $\square$

On the other hand, if instead of using assumption \eqref{K1} we use \eqref{K2}, we are able to prove the following $L^\infty$ estimate involving the distance to the boundary of $\Omega$.
\begin{theorem}[Upper boundary estimates]\label{thm: Upper boundary estimates 2} Let $m>0$ with $m\neq 1$ and $N>2s>\gamma$. Assume \eqref{K2} and \eqref{K}. Then, for any $u$ being a $H^*$-solution  with nonnegative initial datum $u_0\in L^p(\Omega)\cap H^*(\Omega)$ with $p>\frac{N}{2s-\gamma}$ and $p\ge1+m$ it holds that
  \begin{equation}
    \left\|\frac{u(t)}{\mbox{\rm dist}(\cdot,\partial\Omega)^{\gamma/m}}\right\|_\infty\le \ka\frac{\|u_0\|_p^{1/m}}{t^{\frac{\alpha}{m}}}\qquad\forall t>0\,.
  \end{equation}
\end{theorem}
\begin{proof}
  Let us combine Proposition \ref{Fundamental pointwise estimate} with assumption \eqref{K2} in order to obtain
  \begin{align*}
    u^m(t,x_0)&\le\frac{\c}{t^\alpha}\int_\Omega\left(u_0(x)-u(t_1,x)\right)\G(x,x_0)\dx\le \frac{\c}{t^\alpha}\int_\Omega \frac{u_0(x)\,\mbox{\rm dist}(x_0,\partial\Omega)^{\gamma}}{|x-x_0|^{N-2s+\gamma}}\dx\\
    &\le c_\Omega\frac{\|u_0\|_p}{t^\alpha}\,\mbox{\rm dist}(x_0,\partial\Omega)^{\gamma}\,,
  \end{align*}
  where we have used the nonnegativity of $u(t_1,x)$ with $t_1=t$ if $0<m<1$ and $t_1=2t$ if $m>1$ and Hölder's inequality with $p>\frac{N}{2s-\gamma}$. The result follows by dividing both sides of the inequality by $\mbox{\rm dist}(\cdot,\partial\Omega)^{\gamma}$ and taking the essential supremum of $x_0$ in $\Omega$.
\end{proof}

As it is stated in Theorem \ref{Thm Boundary behaviour}, the upper boundary estimate above and the comparison principle allow us to provide a Global Harnack Principle for a class of $H^*$-solutions.
\begin{theorem}[Global Harnack Principle for ``big'' initial data]\label{GHP in section} Let $m>\frac{N-2s}{N+2s}$ with $m\neq 1$ and $N>2s>\gamma$. Assume  \eqref{K2} and \eqref{K}. If the initial datum $u_0\in L^p(\Omega)\cap H^*(\Omega)$ with $p>\frac{N}{2s-\gamma}$ and $p\ge1+m$ satisfies
$$ u_0(x)\ge c\,\mbox{\rm dist}(x,\partial\Omega)^{\gamma/m}\qquad\mbox{ for a.e. }x\in\Omega\,,$$
for certain $c>0$, then it holds that for every $t>0$
  \begin{equation}
     \frac{c_0}{1+t^{\alpha}}\le\frac{ u^m(t,x)}{\mbox{\rm dist}(x,\partial\Omega)^\gamma}\le \frac{c_1}{t^{\alpha}}\qquad\mbox{ for a.e. } x\in\Omega\,.
  \end{equation}
\end{theorem}
\begin{proof}
  The upper bound follows directly by Theorem \ref{thm: Upper boundary estimates 2}. For the lower bound we use the comparison principle of Theorem \ref{Comparison Principle} with respect to the separate-variables solution, which exists if $m>\frac{N-2s}{N+2s}$. This solution has the form
  $$v(t,x)=F(t)S(x)\,,$$
  with $S$ solving the elliptic problem $\A S^m= S$ on $\Omega$ and $F$ solving the fractional ODE of the form $\Dc F=-F^m$ for $t>0$ with $F(0)=F_0>0$. We give a  precise  description of these solutions in Section \ref{section: sharp decay}. As it was proved in \cite{BFV2018CalcVar}, under assumption \eqref{K2} the function $S(x)$ behaves like
$$S(x)\asymp \mbox{\rm dist}(x,\partial\Omega)^{\frac{\gamma}{m}}\qquad\mbox{for a.e. }x\in\Omega\,,$$
and hence, from the assumption on the initial data we have
$$u_0(x)\ge c\,\mbox{dist}(x,\partial\Omega)^{\frac{\gamma}{m}}\ge C_0S(x)\qquad\mbox{ a.e. in }\Omega\,.$$
Now, choosing $F_0=C_0>0$ we obtain
$$u_0(x)\ge F_0S(x)=v(0,x)\qquad\mbox{ a.e. in }\Omega\,.$$
Finally, by comparison principle in Theorem \ref{Comparison Principle}, for a.e. $x\in\Omega$ and every $t>0$ it holds that
$$u(t,x)\ge v(t,x)=F(t)S(x)\ge c\,F(t) \,\mbox{dist}(x,\partial\Omega)^{\gamma/m}\ge c_0\,\frac{\mbox{dist}(x,\partial\Omega)^{\gamma/m}}{1+t^{\alpha}}\,,$$
since we have used the lower estimates for $F(t)$ proved in \cite{VZ}, see Theorem \ref{FODE 1}.
\end{proof}

\section{Sharp time decay of \texorpdfstring{$L^p$}{Lp}-norms}\label{section: sharp decay}
This section contains the proof Theorem \ref{Thm Sharp Decay Lp};  in particular, we show that every $H^*$-solution does not extinguish in finite time.

Let us explain, heuristically, why one should {\em{not}} expect extinction in finite time,  unlike the case with classical time derivative. Let us consider solutions of  for \eqref{CDP} of the form
$$v(t,x)=F(t)S(x).$$ Then, $F$ and $S$ satisfy the relation
\begin{equation*}
  \frac{\A S^m(x)}{S(x)}=-\frac{\Dc F(t)}{F^m(t)}=\lambda\,.
\end{equation*}
The equation satisfied by $S(x)$ is the same as in the classical time derivative case. Defining $V(x):=S^m(x)$ and $p:=\tfrac{1}{m}$ we get the well-known semilinear equation of Lane-Emden-Fowler type:
\begin{equation} \label{49}
\begin{cases}
\A V(x)&=\lambda V^p(x)\qquad\mbox{in}\quad\Omega\,,\\
V(x)&=0\qquad\qquad\;\;\mbox{on}\quad\partial\Omega\,.
\end{cases}
\end{equation}
  Due to the Pohozaev identity, in the local and the nonlocal setting, nontrivial solutions to (\ref{49}) fail to exists for $p<\frac{N+2s}{N-2s}$, equivalently, for $m>m_s:=\frac{N-2s}{N+2s}$. As shown in \cite{BFV2018CalcVar}, under assumption \eqref{K2}, solutions to (\ref{49}) behave like
$$V(x)\asymp \mbox{\rm dist}(x,\partial\Omega)^\gamma\asymp\Phi_1(x)\,,$$
with the parameter $\gamma>0$ depending on the operator $\A$.

The time-dependent function $F(t)$, solution to
\begin{equation}
\begin{cases}
\Dc F(t)=-\lambda F^m(t)\,,\\
F(0)=F_0>0\,,
\end{cases}
\end{equation}
has a completely different behaviour than in the classical case: by Theorem \ref{FODE 1}, proved in \cite{VZ}, we know that $F(t)$ is bounded above and below by a function of the form $\frac{1}{1+t^{\frac{\alpha}{m}}}$, more precisely
\begin{equation}\label{Decay of time dependent}
\frac{c_1}{1+t^{\frac{\alpha}{m}}}\le F(t)\le\frac{c_2}{1+t^{\frac{\alpha}{m}}} \qquad \forall t>0\,,
\end{equation}
for some constants $c_1,c_2>0$ that only depend on $F_0,\lambda,\alpha$ and $m$. The lower bound in \eqref{Decay of time dependent} implies that $v(x,t)$ does not vanish in finite time. Note also that the decay rate of any $L^p$-norm of $v(x,t)$  coincides with the regularization rate shown in Theorem \ref{Thm smoothing effect}.

An important question is whether this non-extinction phenomenon in finite time occurs for all nonnegative solution of \eqref{CDP}. As we will prove in this section, the answer is affirmative and the time decay for the $L^p$-norms for all  $1\le p\le\infty$ is sharp.

The proof starts with a lower bound for the $L^1_{\Phi_1}$-norm of the solution with  a rate of $t^{-\alpha/m}$. Then, we conclude with an interpolation argument between the $L^1_{\Phi_1}$-norm and the $L^\infty$-norm.

But first let us show the following estimate for the $L^1_{\Phi_1}$-norm of the Caputo derivative.
\begin{lemma}
  Let $m>0$ and $u$ be a $H^*$-solution of \eqref{CDP} with $u_0\in L^{1+m}(\Omega)\cap H^*(\Omega)$. Then, estimate \eqref{Akagi estimate}, which we recall here,
   $$ \ell*\|\Dc u(t)\|^2_{H^*(\Omega)}\in L^\infty((0,T))\,,$$
   implies that
  \begin{equation}\label{estimate for L1Phi1 Caputo}
    \ell*\|\Dc u(t)\|_{L^1_{\Phi_1}(\Omega)}\in L^\infty((0,T))\,.
  \end{equation}
\end{lemma}
\begin{proof}
Since  the kernel $\ell$ is nonnegative and $\Phi_1\in H(\Omega)$, we have that
\begin{align*}
  \left|\ell*\left(\int_\Omega\Dc u(t,x)\Phi_1(x)\dx\right)\right|&=\left|\int_{0}^{t}\int_\Omega\ell(t-s)\Dc u(s,x)\Phi_1(x)\dx\dt\right|\\
  &\le\|\Phi_1\|_{H(\Omega)}\int_{0}^{t}\ell(t-s)\|\Dc u(s)\|_{H^*(\Omega)}\ds\\
  &\le\frac{\|\Phi_1\|_{H(\Omega)}}{2}\int_{0}^{t}\ell(t-s)\left(1+\|\Dc u(s)\|_{H^*(\Omega)}^2\right)\ds\\
  &\le\frac{\|\Phi_1\|_{H(\Omega)}}{2}\left(\|\ell\|_{L^1(0,T)}+\ell*\|\Dc u(t)\|_{H^*(\Omega)}^2\right)\,,
\end{align*}
 thanks to  Young's inequality  in the third line. We conclude by taking the essential supremum in $t\in(0,T)$ and using estimate \eqref{Akagi estimate}.
\end{proof}

The last lemma provides a lower and upper bounds for the $L^1_{\Phi_1}$-norm of the $H^*$-solutions.
\begin{proposition}\label{L1Phi decay}
  Let $0<m\le1$ and u be a $H^*$-solution of \eqref{CDP} with $0\le u_0\in L^1_{\Phi_1}(\Omega)$. Then,
  \begin{equation}\label{nonvanishing}
    \frac{c}{1+t^{\frac{\alpha}{m}}}\le\|u(t)\|_{L^1_{\Phi_1}}\le \|u_0\|_{L^1_{\Phi_1}(\Omega)}\qquad\quad\forall t>0\,,
  \end{equation}
  with  $c$ depending on $\alpha,m,\lambda_1,\|\Phi_1\|_{L^1(\Omega)}$ and $\|u_0\|_{L^1_{\Phi_1}}$.

   \noindent Moreover, if $0\le u_0\in L^p(\Omega)\cap H^*(\Omega)$ with $p>\frac{N}{2s}\ge 1+m$ or $p\ge 1+m>\frac{N}{2s}$, with $m>1$, then
   \begin{equation}
     \frac{c_0}{1+t^{\frac{\alpha}{m}}}\le\|u(t)\|_{L^1_{\Phi_1}}\le\|u_0\|_{L^1_{\Phi_1}(\Omega)}\qquad\quad\forall t>0\,.
   \end{equation}
\end{proposition}

\begin{proof}
 Inspired by  \cite{VZ-Blowup}, we look for a fractional differential inequality for the $L^1_{\Phi_1}$-norm of the solution and then use a comparison argument.

\noindent$\bullet~$\textbf{Step 1. }{\it Upper bound.} From the definition of the nonnegative $H^*$-solution $u$ in \eqref{strong H solution}, we deduce
\begin{align*}
  \int_\Omega \partial_t[k*(u-u_0)](t)\Phi_1\dx=-\lambda_1\int_{\Omega}u^m(t)\Phi_1\dx\le 0\qquad\mbox{ for a.e. }t>0\,,
\end{align*}
using as a test function $0\le\lambda_1\Phi_1\in H(\Omega)$. 
We now convolve both sides of the equation with the nonnegative kernel $\ell$ and obtain
\begin{align*}
  \ell*\left(\int_\Omega \partial_t[k*(u-u_0)](t)\Phi_1\dx\right)=\int_\Omega \left(\ell*\partial_t[k*(u-u_0)]\right)(t)\Phi_1\dx\le 0\,.
\end{align*}
Since $\ell*k=1$ and $\partial_t\left[1*(u(t,x)-u_0(x))\right]=u(t,x)-u_0(x)$ for a.e. $x\in\Omega$, we conclude that
$$\int_\Omega \left(u(t)-u_0)\right)\Phi_1\dx\le 0\qquad\forall t>0\,.$$

\noindent$\bullet~$\textbf{Step 2. }{\it Lower regularized differential inequality}. We first regularize the kernel $k$ to apply a comparison principle for fractional ODEs. We know from  \eqref{Akagi estimate} that $Y(t):=\|u(t)\|_{L^1_{\Phi_1}}\in L^\infty((0,T))$ and $k*(u-u_0)\in H^1([0,T])$. Moreover, there exists a  nonnegative kernel $h_n\in H^1([0,T])$ (see for instance \cite{VZ,WWZ}) satisfying
\begin{align*}
  h_n*g\rightarrow g \qquad\mbox{in } L^p(0,T)\quad\mbox{when }n\rightarrow\infty\,,
\end{align*}
for any function $g\in L^p((0,T))$.  The aim of this step is to prove the following inequality for the function $Y\in L^\infty((0,T))$:
\begin{equation}\label{Regularized fractional ineq}
  \partial_t[k_n*(Y-Y_0)](t)\ge -h_n*f(t,Y(t))\qquad\mbox{for a.e. } t\in(0,T)\,,
\end{equation}
with $k_n:=h_n*k$ and $f(t,x)$ nondecreasing in $x$ for a.e. $t>0$ and $f(t,Y(t))\in L^1((0,T))$.

 We start with the definition of the $H^*$-solution $u$ in \eqref{strong H solution} with the admissible test function $\lambda_1\Phi_1\in H(\Omega)$,
\begin{equation}\label{testing with Phi1}
        \int_\Omega \partial_t[k*(u-u_0](t)\Phi_1\dx=-\lambda_1\int_\Omega u^m(t)\Phi_1\dx\qquad\mbox{ for a.e. } t>0.
      \end{equation}
\begin{enumerate}[label=\roman*)]
  \item {\it Case $m=1$}: We convolve both sides of \eqref{testing with Phi1} with the nonnegative kernel $h_n\in H^1([0,T])$ and obtain for a.e. $t>0$
  \begin{equation*}
        \int_{0}^{t}\int_\Omega h_n(t-\tau)\partial_t[k*(u-u_0)](\tau)\Phi_1\dx\dtau=
        -\lambda_1\int_{0}^{t}\int_\Omega h_n(t-s) u(\tau)\Phi_1\dx\dtau\,.
      \end{equation*}
      Then, choosing $f(t,x)=\lambda_1 x$ we deduce \eqref{Regularized fractional ineq} from Fubini-Tonelli's Theorem and the regularity of $h_n\in H^1([0,T])$.
  \item {\it Case $0<m<1$}: We use Hölder's inequality in \eqref{testing with Phi1} and get
      \begin{equation*}
        \int_\Omega \partial_t[k*(u-u_0](t)\Phi_1\dx\ge -\lambda_1\|\Phi_1\|_1^{1-m}\left(\int_\Omega u(t)\Phi_1\dx\right)^m\qquad\mbox{ for a.e. } t>0.
      \end{equation*}
      Then, we convolve the inequality above with the nonnegative kernel $h_n\in H^1([0,T])$,
      \begin{equation*}
        \int_{0}^{t}\int_\Omega h_n(t-\tau)\partial_t[k*(u-u_0)](\tau)\Phi_1\dx\dtau
        \ge-\lambda_1\|\Phi_1\|_1^{1-m}\!\int_{0}^{t}h_n(t-s)\left(\int_\Omega u(\tau)\Phi_1\dx\right)^m\dtau,
      \end{equation*}
      for a.e. $t>0$. We deduce \eqref{Regularized fractional ineq} with $f(t,x)=\lambda_1\|\Phi_1\|_1^{1-m}|x|^{m-1}x$ using Fubini-Tonelli's Theorem and the regularity of the kernel $h_n\in H^1([0,T])$.
  \item {\it Case $m>1$}: We use smoothing effects of Theorem \ref{Thm smoothing effect} in \eqref{testing with Phi1}
  \begin{align*}
    \int_\Omega \partial_t[k*(u-u_0](t)\Phi_1\dx&=-\lambda_1\int_\Omega u^m(t)\Phi_1\dx\ge-\lambda_1\|u(t)\|_\infty^{m-1}\int_\Omega u(t)\Phi_1\dx\\
    &\ge-\lambda_1\ka\frac{\|u_0\|_p^{\frac{m-1}{m}}}{t^{\frac{\alpha(m-1)}{m}}}\int_\Omega u(t)\Phi_1\dx\,.
  \end{align*}
  Hence, convolving with the nonnegative kernel $h_n\in H^1([0,T])$ we obtain \eqref{Regularized fractional ineq} as before, but with $f(t,x)=\lambda_1\ka\|u_0\|_p^{\frac{m-1}{m}} \,t^{-\frac{\alpha(m-1)}{m}}x$.
\end{enumerate}
\noindent$\bullet~$\textbf{Step 3. }{\it Lower  regularized integral inequality}. Now, we combine the inequality for solution $Y$ of \eqref{Regularized fractional ineq} with the solution $Z$ of the equation
\begin{equation}\label{FODE of Z}
  \partial_t[k_n*(Z-Z_0)](t)= -h_n*f(t,Z(t))\qquad\mbox{for a.e. } t\in(0,T)\,,
\end{equation}
with $Z_0=Y_0$ and we prove that
\begin{equation}\label{Regularized integral inequality}
  \frac{1}{2}[k_n*(Z-Y)_+^2](T) +\int_{0}^{T}(Z(t)-Y(t))_+h_n*(f(t,Z(t))-f(t,Y(t)))\dt\le 0\,.
\end{equation}

For this purpose, we subtract the equation of $Y$ in \eqref{Regularized fractional ineq}from \eqref{FODE of Z} and we multiply  the resulting equation by $(Z-Y)_+\in L^\infty(0,T)$. This yields
$$(Z-Y)_+(t)\partial_t\left[k_n*(Z-Y)\right](t)+(Z-Y)_+ (t)\, h_n*(f(t,Z(t))-f(t,Y(t)))\le 0\,.$$
By the fundamental identity of Lemma \ref{fundamental identity} with $h(r)=\frac{1}{2}(r_+)^2$, we obtain
\begin{align*}
  &\frac{1}{2}\partial_t\left[k_n*(Z-Y)_+^2\right](t)+(Z-Y)_+(t)\, h_n*(f(t,Z(t))-f(t,Y(t)))\\
  \le&(Z-Y)_+(t)\partial_t\left[k_n*(Z-Y)\right](t)+(Z-Y)_+ (t)\, h_n*(f(t,Z(t))-f(t,Y(t)))\le 0.
\end{align*}
Integrating in $[0,T]$ and recalling that $Y_0=Z_0$ a.e. in $\Omega$, we obtain \eqref{Regularized integral inequality}\,.

\noindent$\bullet~$\textbf{Step 4. }{\it Comparison with $Z$.} We want to take the limit in \eqref{Regularized integral inequality} and obtain
\begin{equation}\label{convolution negative}
  k*(Z-Y)_+^2(T)\le 0\,,
\end{equation}
 which implies $Z(t)\le Y(t)$ for a.e. $t>0$ since $k(t)>0$ for all $t>0$.

 Recalling that $(Z-Y)_+^2\in L^\infty((0,T))$, we deduce
\begin{align*}
 \left| k_n*(Z-Y)_+^2(T)-k*(Z-Y)_+^2(T)\right|\le \|(Z-Y)_+^2\|_\infty\int_{0}^{T}|(h_n*k)(t)-k(t)|\dt\rightarrow 0\,,
\end{align*}
 when $n\rightarrow\infty$. On the other hand, since $f(t,Y(t)),f(t,Z(t))\in L^1((0,T))$ we obtain
 \begin{align*}
    &\left| \int_{0}^{T}\!\!(Z(t)-Y(t))_+h_n*(f(t,Z(t))-f(t,Y(t)))\dt -\int_{0}^{T}(Z(t)-Y(t))_+(f(t,Z(t))-f(t,Y(t)))\dt\right|\\
    \le& \|(Z-Y)_+\|_\infty\int_{0}^{T}\!\! |[h_n*(f(t,Z(t))-f(t,Y(t)))](t)-(f(t,Z(t))-f(t,Y(t)))(t)|\dt\rightarrow 0\,,
 \end{align*}
when $n\rightarrow\infty$. Thus, taking the limit in \eqref{Regularized integral inequality} we obtain
\begin{equation*}
  \frac{1}{2}[k*(Z-Y)_+^2](T) +\int_{0}^{T}(Z(t)-Y(t))_+(f(t,Z(t))-f(t,Y(t)))\dt\le 0\,.
\end{equation*}
Notice that $(Z(t)-Y(t))_+(f(t,Z(t))-f(t,Y(t)))\ge 0$ for a.e. $t>0$ by the nondecreasing assumption on $f$, and hence \eqref{convolution negative} holds. The result follows from \eqref{convolution negative} and  the fact that $k(t)>0$ for every $t>0$.

\noindent$\bullet~$\textbf{Step 5. }{\it Lower bound.} Finally, we show that if $Z_0=Y_0$, then
$$\frac{c}{1+t^{\frac{\alpha}{m}}}\le Z(t)\le Y(t)\qquad\forall t>0\,,$$
which give us the lower bound for $Y(t)=\|u(t)\|_{L^1_{\Phi_1}(\Omega)}$.

Let us recall the form of  $Z$ for the three particular case of the fractional ODEs in \eqref{FODE of Z}:
\begin{enumerate}[label=\roman*)]
  \item If $m=1$ and $f(t,x)=cx$, $Z$ is the Mittag-Leffler function $Z(t)=Y_0E_\alpha[-ct^\alpha]$, see \eqref{Mittag-Lefler function}.
  \item If $0<m<1$ and $f(t,x)=c|x|^{m-1}x$, we do not know the explicit solution of $Z$, but we know its behaviour by the result of Vergara and Zacher in \cite{VZ}, see Theorem \ref{FODE 1}.
  \item  If $m>1$ and $f(t,x)=ct^{-\frac{\alpha(m-1)}{m}}x$, $Z$ is the so-called Kilbas-Saigo function of the form $Z(t)=Y_0E_{\alpha,\frac{1}{m},\frac{1-m}{m}}[-c t^{\frac{\alpha}{m}}]$, see Lemma \ref{Lem Kilbas-Saigo}.
\end{enumerate}
Furthermore, by \eqref{Mittag-Leffler decay}, Theorem \ref{FODE 1} and Proposition \ref{decay of Kilbas FODE}, the three functions $Z$ defined above have the same decay
$$Z(t)\asymp\frac{1}{1+t^{\frac{\alpha}{m}}}\,,$$
which implies the desired lower bound
\end{proof}

Now that we have a lower bound for the decay of the  $L^1_{\Phi_1}$-norm of solutions of \eqref{CDP}, we are in a position to prove the sharp time decay of the $L^p$-norm stated in Theorem \ref{Thm Sharp Decay Lp}.\vspace{3mm}

\noindent{\it Proof of Theorem \ref{Thm Sharp Decay Lp}.} The proof follows by an interpolation argument, the lower bound of the $L^1_{\Phi_1}$-norm, and the $L^p-L^\infty$ smoothing effects. By Proposition \ref{L1Phi decay} and Theorem \ref{Thm smoothing effect} we obtain that for every $t>0$
\begin{align}\label{ineq for Lp norms}
  \frac{c}{1+t^{\frac{\alpha}{m}}} &\le \|u(t)\|_{L^1_{\Phi_1}}
  \le \|u(t)\|_p\;\|\Phi_1\|_{p'} \le|\Omega|\|u(t)\|_\infty\,\|\Phi_1\|_\infty
  \le|\Omega|\,\|\Phi_1\|_\infty\,\ka\, \frac{\|u_0\|^{\frac{1}{m}}_{p_0}}{t^{\frac{\alpha}{m}}}\,,
\end{align}
for any $1\le p\le\infty$ and $p'$ being its conjugate exponent.
Hence, for the case $1\le p<\infty$, we divide \eqref{ineq for Lp norms} by $\|\Phi_1\|_{p'}$ and for the case $p=\infty$ we divide inequality \eqref{ineq for Lp norms} by $|\Omega|\,\|\Phi_1\|_\infty$ and conclude the proof.

\hfill$\square$

\section{Open questions}
In the study of the problem for the Caputo Porous Medium equation many questions have arisen. We share the most important ones with the community.
\begin{enumerate}[label=\roman*),wide, labelindent=0pt]
  \item {\bf T-contractivity:} From the fractional gradient flow theory developed in \cite{GA,LiLiu2019,LiSalgado} we know that the contractivity in the Hilbert space  $H^*(\Omega)$ holds, that is, if $u$ and $v$ are $H^*$-solutions with initial data $u_0,v_0$, respectively, then
      $$\|u(t)-v(t)\|_{H^*(\Omega)}\le\|u_0-v_0\|_{H^*(\Omega)}\qquad\forall t>0\,.$$
      In order to obtain the comparison principle of Theorem \ref{Comparison Principle}, we try to prove  a stronger version of property above called T-contractivity in $H^*(\Omega)$, namely,
      $$\|(u(t)-v(t))_+\|_{H^*(\Omega)}\le\|(u_0-v_0)_+\|_{H^*(\Omega)}\qquad\forall t>0\,.$$
      Since nonnegative functions of $H^*(\Omega)$ are in $L^1_{\Phi_1}(\Omega)$ and $\Phi_1(x)\asymp\mbox{dist}(x,\partial\Omega)^\gamma$, this T-contractivity would imply the comparison principle of Theorem \ref{Comparison Principle} directly. But, due to the lack of regularity estimates for $\Dc u$, we have to prove the T-contractivity in $L^1_{\Phi_1}(\Omega)$ at the discrete level in Theorem \ref{Discrete T-contractivity L1Phi1} and inherit the comparison principle from the discrete comparison principle.

      We think that it would be enough to have $|\Dc u|\in L^2((0,T):H^*(\Omega))$ or $|\Dc u|\in L^2((0,T):L^1_{\Phi_1}(\Omega))$ in order to prove T-contractivity in $H^*(\Omega)$ or $L^1_{\Phi_1}(\Omega)$, respectively. Let us remark that T-contractivity in $L^1_{\Phi_1}(\Omega)$ is also an open problem in the case of classical time derivative, when considering  weak dual solutions of the Porous Medium and Fast Diffusion equation. In \cite{BII,BV2016} it was proved only for \textit{minimal} weak dual solutions.

\item {\bf Higher regularity:}
The classical way to prove Hölder continuity is via some kind of Harnack inequalities for the solution. The problem with PDEs involving Caputo time derivative is that obtaining lower bounds for solutions is significantly more challenging.
In this manuscript we prove a Global Harnack Principle in Theorem \ref{Thm Boundary behaviour2} only for a class of solutions with ``big'' initial data, but we did not succeed to prove it for general initial data.  To our knowledge, the only result in this direction appeared in \cite{ACV} and \cite{ACV-PME}, where the authors apply the De Giorgi's method to show Hölder regularity.

\item {\bf Asymptotic behaviour:} Theorem \ref{Thm Sharp Decay Lp} shows that solutions of \eqref{CDP} does not vanish in finite time for any $m>0$. One may wonder about the behaviour of the solutions when $t\rightarrow\infty$. From the smoothing effects of Theorem \ref{Thm smoothing effect}, we deduce that solutions vanish at time increases, but we would like to know if there is any asymptotic profile $S(x)$ such that
    $$\lim\limits_{t\rightarrow\infty}\|t^{\frac{\alpha}{m}}u(t)-S\|_X=0\,,$$
    for certain norm $X$. From the classical theory, we expect that $S$ is the solution of the elliptic problem described in Section \ref{section: sharp decay} and that the solution behaves at infinity like the the separate-variable solution.
    In the classical approach, the results relies on a change of variables;  with Caputo time derivative this approach would be more complicated, see the fundamental identity for regular kernels in Lemma \ref{fundamental identity}.
\end{enumerate}


\newpage

\appendix

\section{Appendix. Fractional ODEs}\label{sec: FODE}

We collect  some known results about the Caputo derivative and the associated fractional ODEs that are useful for our study. Let us start with the linear differential equation with $\alpha\in(0,1)$, $v_0>0$ and $\lambda\in\mathbb{R}^N$
\begin{equation}\label{Mittag-Lefler function}
  \begin{cases}
    \Dc v(t)=\lambda v(t)\quad\mbox{in }(0,+\infty)\\
    v(0)=v_0,
  \end{cases}
\end{equation}
which solution is the so-called Mittag-Leffler function of the form $v(t)=v_0\,E_\alpha(\lambda t^\alpha)$, where
\begin{equation*}
  E_\alpha(x)=\sum_{k=1}^{\infty}\frac{x^k}{\Gamma(k\alpha+1)}\,.
\end{equation*}
There is a large amount of literature about this function, cf \cite{Diethelm,GorenfloMainardi-Book}. But, let us remark that the behaviour at the infinity of $E_\alpha(-x)$ is given by
\begin{equation}\label{Mittag-Leffler decay}
  E_\alpha(-x)\asymp \frac{1}{1+x}\,,
\end{equation}
as it is shown in \cite{Mainardi2020}.
Moreover, we can see the Mittag-Leffler function $E_\alpha(-x^2)$ with $\alpha\in(0,1)$ as an interpolation between the Gaussian and the polynomial with quadratic decay
$$E_1(-x^2)=e^{-x^2}\le E_{\alpha_1}(-x^2)\le E_{\alpha_2}(-x^2)\le\frac{1}{1+x^2}\,,$$
with $0<\alpha_1\le\alpha_1<1$.
Now, let us consider another linear fractional ODE which has an explicit solution.
\begin{lemma}\cite[Example 4.11]{KST-Book}\label{Lem Kilbas-Saigo}
    Let $0<\alpha<1$ and $\beta>-\alpha$. Then, the solution of
    \begin{equation}\label{Kilbas-Saigo FODE}
        \begin{cases}
          \Dc v(t) &=-\lambda t^\beta v(t)  \\
          v(0)&=v_0,
        \end{cases}
    \end{equation}
    is given by
    $$v(t)=v_0 \,E_{\alpha,1+\frac{\beta}{\alpha},\frac{\beta}{\alpha}}[-\lambda t^{\alpha+\beta}]\,.$$
  \end{lemma}
  This function is called the Kilbas-Saigo function and its explicit form can be found in \cite{BoudasaSimon}. However, for our study we only need the asymptotic behaviour that we state below.
  \begin{proposition}\cite[Proposition 6]{BoudasaSimon}\label{decay of Kilbas FODE}
  Let $0<\alpha<1$, $r>0$ and $l>r-\frac{1}{\alpha}$, then
  $$E_{\alpha,r,l}(-x)\asymp \frac{1}{1+c\,x}\qquad\quad\mbox{as}\quad x\rightarrow\infty\,.$$
  \end{proposition}

In the nonlinear setting, fractional ODEs becomes very delicate. Let us present the following result due to Vergara and Zacher addressing the behaviour of a basic nonlinear fractional ODE.
\begin{theorem}\label{FODE 1}\cite[Theorem 7.1]{VZ}
  Let $\alpha\in(0,1)$, $\lambda,m>0$ and $v_0>0$. Let $v\in H^1_{\rm loc}(\mathbb{R}_+)$ be the solution of
  $$\begin{cases}
      \Dc v(t)=-\lambda v^m(t) & \mbox{in } (0,+\infty)\\
      v(0)=v_0.
    \end{cases}$$
    Then, there exist $c_1,c_2>0$ depending on $v_0, m,\alpha$ and $\lambda$ such that
    \begin{equation}\label{nonlinear ODE decay}
      \frac{c_1}{1+t^{\frac{\alpha}{m}}}\le v(t)\le \frac{c_2}{1+t^{\frac{\alpha}{m}}}\qquad\forall t>0\,.
    \end{equation}
\end{theorem}

The explicit form of the solution to this basic nonlinear fractional ODE is not known, but in \cite{VZ} the authors construct subsolution and supersolution to estimate the lower and upper bounds. Note that the behaviour of $v$ described in \eqref{nonlinear ODE decay} is completely different from the solution of the analogous classical ODE. In the classical case, the value of the exponent $m$ plays a crucial role. Namely, for the classical  ODE of the form $\partial_t v=-\lambda v^m$ with $v(0)=v_0$ we have that:
\begin{itemize}
  \item If $0<m<1$, the solution vanish in finite time since $$v(t)=(v_0^{1-m}-\lambda(1-m)\,t)_+^{\frac{1}{1-m}}\,.$$
  \item If $m=1$, the solution decays exponentially, $v(t)=v_0\exp(-\lambda t)$.
  \item If $m>1$, the solution decays polynomially
$$v(t)=\frac{1}{\left(v_0^{-(m-1)}+\lambda(m-1)t\right)^{\frac{1}{m-1}}}\le\frac{C}{t^{\frac{1}{m-1}}} \,.$$
\end{itemize}
 Notice,  at the level of ODEs, the difference of the asymptotic profile between classical and fractional time derivative. Moreover, there is no continuity of results when $\alpha\rightarrow 1$.\vspace{2mm}

 To conclude this appendix, we present a fundamental identity for fractional derivatives with regular kernels. This formula plays the role of the chain rule $\partial_t h(u(t))=h'(u(t))\partial_t u(t)$ in the classical sense. 

 \begin{lemma}\cite[Lemma 2.2]{VZ}\label{fundamental identity}
   Let $T>0$ and $I$ an open set of $\mathbb{R}$. Further, let $\tilde{k}\in W^{1,1}([0,T])$, $h\in C^1(I)$ and $u\in L^1((0,T))$ with $u(t)\in I$ for a.e. $t\in(0,T)$. Suppose that the functions $h(u),h'(u)u$ and $h'(u)(\tilde{k}'*u)$ belong to $L^1((0,T))$ (which is the case if $u\in L^\infty((0,T))$). Then, we have for a.e. $t\in(0,T)$
   \begin{equation*}
     \begin{split}
        h'(u(t))\frac{\rm d}{\dt}(\tilde{k}*u)(t) =&\frac{\rm d}{\dt}(\tilde{k}*h(u))(t)+\bigg(-h(u(t))+h'(u(t))u(t)\bigg)\tilde{k}(t)  \\
        &+\int_{0}^{t}\bigg(h(u(t-\tau))-h(u(t))-h'(u(t))\left[u(t-\tau)-u(t)\right]\bigg)(-\tilde{k}'(\tau))\dtau\,.
     \end{split}
   \end{equation*}
 \end{lemma}
 Let us remark that when the function $h$ is convex and the kernel $\tilde{k}$ is nonnegative and nonincreasing, we obtain the inequality
 $$\frac{\rm d}{\dt}(\tilde{k}*h(u))\le h'(u(t))\frac{\rm d}{\dt}(\tilde {k}*u)(t)\,.$$

\end{document}